\documentclass[11pt]{article}
\usepackage[utf8]{inputenc}
\date{}

\usepackage{latexsym, color, amssymb, amsfonts,amsmath}
\usepackage{graphicx}
\usepackage{mathrsfs}
\usepackage{latexsym}
\usepackage{enumerate}
\usepackage{verbatimbox}

\usepackage{latexsym, color, amssymb, amsfonts,amsmath}
\usepackage{graphicx}
\usepackage{mathrsfs}
\usepackage{latexsym}
\usepackage{enumerate}
\usepackage{verbatimbox}

\usepackage{amsmath}
\usepackage{amsfonts}
\usepackage{amssymb}
\usepackage{color}
\usepackage{tikz}
\usetikzlibrary{arrows,backgrounds,positioning,fit,shapes,shapes.misc}

\usepackage{euscript}

\usepackage{mathrsfs}
\usepackage{multicol}

\usepackage{euscript}

\usepackage{amssymb}
\usepackage{mathrsfs}
\usepackage{multicol}

\newenvironment{dem}{\noindent\bf Proof.\rm}{\hfill $\negr{\Box}$}

\def\NN{\mbox{{\sl I}}\!\mbox{{\sl N}}}

\newtheorem{proposition}{\bf Proposition}[section]

\newtheorem{theorem}[proposition]{\bf Theorem}
\newtheorem{coro}[proposition]{\bf Corollary}
\newtheorem{remark}[proposition]{\bf Remark}
\newtheorem{definition}[proposition]{\bf Definition}
\newtheorem{defi}[proposition]{\bf Definition}
\newtheorem{prop}[proposition]{\bf Proposition}
\newtheorem{lem}[proposition]{\bf Lemma}
\newtheorem{theo}[proposition]{\bf Theorem}

\newenvironment{proof}{\noindent\bf Proof.\rm}{\hfill $\negr{\Box}$}

 \relax

\newcommand{\negr}[1]{\boldsymbol{#1}}

\date{}

\title{Sentential logics based on $k$-cyclic modal pseudocomplemented De Morgan algebras}

\author{Aldo Figallo-Orellano, Miguel Per\'ez-Gaspar  and  Juan Manuel Ram\'irez-Contreras}
\begin{document}



\maketitle

\thispagestyle{empty}

\begin{abstract}
The study of the theory of operators over modal pseudocomplemented De Morgan algebras was begun in  the papers \cite{AFO1} and \cite{AFO2}. In this paper, we introduce and study the class of   modal pseudocomplemented De Morgan algebras enriched by an automorphism $k$-periodic (or ${\cal C}_k$-algebras) where $k$ is a positive integer; for $k=2$ the class coincides with the one  studied in \cite{AFO1} where the automorphism works as a new unary operator. In the first place, we prove the class ${\cal C}_k$-algebras is a semisimple variety and we determine the generating algebras. Afterwards, we calculate the cardinality of the free ${\cal C}_k$-algebra with $n$ generator. After the algebraic study, we built two sentential logics that have as algebraic counterpart the class of ${\cal C}_k$-algebras that we denote $\mathbb{L}_{k}^{\leq}$ and $\mathbb{L}_{k}$ for every $k$. $\mathbb{L}_{k}$ is a $1$-assertional logic and  $\mathbb{L}_{k}^{\leq}$ is the degree-preserving logic both associated with the class of  ${\cal C}_k$-algebras. Working over these logics, we prove that $\mathbb{L}_{k}^{\leq}$ is paraconsistent,  which is protoalgebraic and finitely equivalential but not algebraizable. In contrast, we prove that  $\mathbb{L}_{k}$ is algebraizable,  sharing the same theorem with  $\mathbb{L}_{k}^{\leq}$, but not paraconsistent.

\end{abstract}

\section{Introduction and Preliminaries}\label{preli}

In 1978, A. Monteiro introduced tetravalent modal algebras (or TMA-algebras for short) as algebras $\langle L,\wedge, \vee, \sim, \nabla, 0, 1 \rangle$ of type $(2, 2, 1, 1, 0, 0)$ such that $\langle L,\wedge, \vee, \sim, 0, 1 \rangle$ are De Morgan algebras which satisfy the following conditions:

\begin{center} $\nabla x \vee \sim x=1$,

 $\nabla x \wedge \sim x= \sim x \wedge x$. 
\end{center}

These algebras arise as a generalization of three--valued {\L}ukasiewicz algebras by omitting the identity $\nabla (x \wedge y)= \nabla x \wedge \nabla y$. The variety of TMAs is generated by the well-known four-element De Morgan algebra expanded with a simple modal
operator $\nabla$ (i.e., $\nabla 1=1$ and $\nabla x=0$ for $x\not=1$. Besides, $\sim 0=1$ and $\sim x= x$ for $x\not=0,1$). These algebras were studied by  I. Loureiro, A. V. Figallo and A. Ziliani, see \cite{FL,FZ1} for instance.

Later, J. Font and M. Rius introduced some logics that have as algebraic counterpart to the class of TMA-algebras. In particular, they defined the logic ${\cal TML}$ which can be seen as the degree-preserving logic associated with the class of TMA-algebras, see \cite[Proposition 3.7]{FR} and \cite{MF1}. These authors also proved that the logic ${\cal TML}$ is not algebraizable with  Blok-Pigozzi's method. More recently, M. Figallo proved that ${\cal TML}$  is in fact a paraconsistent logic and a {\em Logic of Formal Inconsistency}  in \cite{MF}. Recall that the  Logics of Formal Inconsistency (LFIs) were introduced by W. Carnielli and J. Marcos (see \cite{CM}), and the Logics of Formal Undeterminedness (LFUs) was introduced in \cite{Marcos}. Other paraconsistent degree-preserving logics were studied by Ertola {\em et al.} in \cite{CE} and \cite{Erto15}; besides, these kinds of logics were built over {\em Distributive Involutive Residuated Lattices} in  \cite{EFFG}, where the authors presented a large family of LFIs and LFUs. 

Another class of algebras related with TMA-algebra are {\em Stone Involutive Algebras} (SI-algebras), these structures are De Morgan algebras enriched with a modal operator similar to $\nabla$ considered above; they were defined  by Cingoli and Gallego in \cite{CG}. Later, L. Cant\'u and M. Figallo studied  the degree-preserving logic associated with the class of  SI-algebras (${\cal SI}^\leq$) providing a sound and complete  Gentzen calculus. Besides,  they proved that ${\cal SI}^\leq$ is a {\em strong} LFI. Later on, S. Marcelino and U. Rivieccio proved that ${\cal SI}^\leq$ is selfextensional and non-protoalgebraic logic; hence, non-algebraizable, \cite{MR}.

On the other hand, A. V. Figallo considered  the subvariety of pseudocomplemented De Morgan algebras which verifies: (tm) $x\vee \sim x \leq x\vee x^\ast $ in \cite{AVF}. This author called them modal pseudocomplemented De Morgan algebras (or $mpM-$algebras). Recall that a pseudocomplemented De Morgan algebra $A$ is a De Morgan algebra with a unary operator $^\ast$  such that every $a\in A$, the element $a^\ast$ is the pseudocomplement of a; i.e. $a\wedge x = 0$ if and only if $x \leq a^\ast$. Later, he showed that every $mpM-$algebra is a TMA by defining $\nabla x=\,\sim(\sim x \wedge x^\ast)$ and $\triangle x = \sim \nabla \sim x$.  In \cite{NO} (see also \cite{NO1}), the authors have proven that the subdirectly irreducible $mpM-$algebras are three as TMAs, in fact, Hasse diagrams  are the same in each case, but $3$-chain-$mpM-$algebra is not a subalgebra of four-elements. The mentioned algebras are the following:  $T_2=\{0, 1\}$ with $0<1$, $\sim 0 =0^{\ast}=1$, $\sim 1 =1^{\ast}=0$; $T_3=\{0,a, 1\}$, with $0 < a <1$, $\sim a = a$, $a^{\ast} =0$, $\sim 0 = 0^{\ast} =1$,  $\sim 1 =1^{\ast} =0$; $T_4=\{0,a, b, 1\}$ with $a \not\leq b$, $b\not\leq a$ and $0< a, b <1$, $\sim b =a^{\ast} =b$,  $\sim a=b^{\ast} =a$, $\sim 0 =0^{\ast} =1$,  $\sim 1 =1^{\ast} =0$. It is worth mentioning that $mpM-$algebras constitute a proper subvariety of the variety ${\cal V}_0$  studied by  H. Sankappanavar in \cite{HS}. More recently, the theory of operators over $mpM-$algebras was considered in \cite{AFO1,AFO2}. In particular, the authors studied the class of  $mpM-$algebras enriched with an automorphism of period $2$, this automorphism works as a new unary operator, \cite{AFO2}.

One of the principal objects of the present paper is  to generalize the work of \cite{AFO2} for an automorphism of period $k$, where $k$ is a positive integer. This kind of operators was studied over other algebraic strucures, for instance, over Post algebras of order $p$ by L\'opez-Martinolich in \cite{Marti} (see also \cite{Marti1,DV11}), where the variety of $k$-cyclic Post algebras of order $p$  was studied. Furthermore, it was proved that  there exists an equivalence between the variety generated by the simple $k$-cyclic Post algebra of order $p$ and  the variety generated by the finite field with $p^k$ elements, see  \cite{Marti1}.

The paper is organized as follows: In Section 2, we introduce the class of ${\cal C}_k$-algebras which is in fact a variety; and, we display some properties of {\em prime spectrum} for a given algebra. Section 3, we present a new implication in order to use A. Monteiro's techniques to prove the variety semisimplicity. Section 4, maximal congruences and simple algebras of the class are characterized; and, as consequence, the simple ${\cal C}_k$-algebras are determined. In Section 5, we describe the free algebra with $n$ generators calculating their cardinality using the algebraic properties displayed before; as a by-product, we check that for $k=2$ and $k=1$ the cardinality coincides with others obtained in the literature. Finally, in Section 6 and 7, we build a family of sentential logics that have as algebraic counterpart the class of ${\cal C}_k$-algebras denoted $\mathbb{L}_{k}^{\leq}$ and $\mathbb{L}_{k}$ for every $k$. The logic $\mathbb{L}_{k}$ is $1$-assertional  and  $\mathbb{L}_{k}^{\leq}$ is the degree-preserving logic, both associated with the class of  ${\cal C}_k$-algebras. Firstly, we prove that $\mathbb{L}_{k}^{\leq}$ is an LFI (hence paraconsistent) and an LFU which is protoalgebraic and finitely equivalential but not algebraizable. Finally, we prove that  $\mathbb{L}_{k}$ is algebraizable which is an LFU and shares the same theorem with  $\mathbb{L}_{k}^{\leq}$ but it is not paraconsistent.

\section{ The class of $k$-cyclic $mpM$-algebras: ${\cal C}_k$-algebras}

In this section, we will introduce a new class of algebras that we call  ${\cal C}_k$-algebras which is in fact an equational one. Later, we will display  algebraic properties with the purpose to determine the generating algebras of the variety. First, let us consider the following definition:

\begin{defi}
A $k$-cyclic $mpM$-algebra {\rm (}for short ${\cal C}_k$-algebra{\rm )} is a pair  $(A,t)$ where  $A$ is an  $mpM$-algebra and the function  $t:A\to A$ is an automorphism  such that $t^k (x)=x$ where $k$ is an integer $k\geq 0$. Besides, we write $t^0(x)=x$ and $t^{n}(x) = (t^{n-1}\circ t)(x)$ if $n\geq 1$.  
\end{defi}

Sometimes we write ``$tx$'' instead of ``$t(x)$''.  It is worth mentioning that the class ${\cal C}_k$-algebra is a variety; and, as examples, we have the classes  $1$-cyclic $mpM$-algebra  and  $2$-cyclic $mpM$-algebra were studied in \cite{NO} and \cite{AFO1}, respectively.  Besides,  if an  $mpM$-algebra  is a $3$-valued \L ukasiewicz algebra then the $k$-cyclic version is a special class studied in \cite{Marti,Marti1,DV11}. Now, we present a Lemma that will be used in the rest of the paper. 

\begin{lem}{\rm(\cite{AFO1})}
In every ${\cal C}_k$-algebra, the following conditions hold: \\[1.5mm]
\begin{tabular}{ll}
{\rm (T1)} \, $\triangle 0=0$, $\triangle 1=1$, & {\rm (T2)} \, $\triangle x\leq x$, $x\leq \nabla x$,\\[1mm]
{\rm (T3)} \, if $x \leqslant y$ then $\triangle x\leq \triangle y$,& {\rm (T4)} \, $\triangle x$ is a Boolean element,\\[1mm]
{\rm (T5)} \,  $(\sim\triangle x)^\ast=  \triangle x$, & {\rm (T6)} \,  $(\sim\triangle x)^\ast=  \triangle (\sim x)^\ast$,\\[1mm]
{\rm (T7)} \, $\triangle\triangle x= \triangle x $, $\triangle \nabla x = \nabla x$, $ \nabla \triangle  x = \triangle x$, & {\rm (T8)} \, $  (\triangle x)^\ast= \sim \triangle x$, \\[1mm]
{\rm (T9)} \, $\triangle \sim \triangle x=\,\sim \triangle x$, & {\rm (T10)} \, $\triangle (x \wedge y)= \triangle x \wedge \triangle y$, \\[1mm]
{\rm (T10)(bis)} \,  $\nabla (x \vee y)= \nabla x \vee \nabla y$, & {\rm (T11)} \,  $x\in \triangle A$ iff $x= \triangle x$ iff $x= \nabla x$, \\[1mm]
{\rm (T12)} \, $\triangle A$ is an $S$-subalgebra of $A$, & {\rm (T13)} \, $\sim x \wedge \triangle x = 0$ {\rm (}or  $ x \vee \nabla \sim x = 1${\rm )},\\[1mm]
{\rm (T14)} \, $x\vee \sim \triangle x =1$, & {\rm (T15)} \, $\triangle( \triangle x \vee y) = \triangle x \vee \triangle y$, \\[1mm]
{\rm (T16)} \,  $\triangle( \sim \triangle x \vee y) = \sim \triangle x \vee \triangle y$, & {\rm (T17)} \, $\nabla (\triangle x \wedge y ) = \triangle x \wedge \nabla y$,\\[1mm]
{\rm (T18)} \, $\triangle( \nabla x \vee  y ) = \nabla x \vee \triangle y$, & 
\end{tabular}
\end{lem}

\

The relation between the congruences and certain filters is given by the following Remark.

\begin{remark}\label{Ockc}
Taking into account the studied of congruences of  \cite[Theorem 2.6]{AFO1}, we can affirm that the  $c$-filters (i.e. $t(\triangle F)=F$ where $F$ is a filter)  characterize the congruences of given  ${\cal C}_k$-algebra, moreover:

\begin{itemize}
\item[{\rm(i)}] The relation $Con (A)= \{R(F): F\in {\cal C} (A)\}$ where $R(F) = \{(x,y) \in A\times A:$ there is $f \in F$ such that $x\wedge f= y \wedge f \}$ is a congruence, we denote by  $ {\cal C} (A)$ the set of all $c$-filters of $A$.

\item[{\rm (ii)}] The posets $Con (A)$ and  $ {\cal C} (A)$ are isomorphic where $Con (A)$  is the poset of all congruences of $A$.
\end{itemize}
\end{remark}

Now, we give a result to be used to characterize the prime spectrum for a given  ${\cal C}_k$-algebra. First, recall that the notion of prime filter, ultrafilter, maximal and minimal are in the usual ones, see \cite{RB.PD}. Also, we are going to use the well-known Birula--Rasiowa transformation $\varphi$, see, for instance, \cite{FR}. Recall that, if $A$ is a De Morgan algebra, the map $\varphi$ is defined as follows: For every prime filter $P$ of $A$ 

$$\varphi(P)=A \backslash \sim P=A \backslash \{\sim x: x\in P\}$$

This map has the following properties:

\begin{itemize}
\item[$\bullet$] $\varphi(P)$ is a prime filter of $A$,
\item[$\bullet$] $\varphi(\varphi(P))= P$, 
\item[$\bullet$] if $Q$ is a prime filter of $A$ such that $P\subseteq Q$ then $\varphi(Q)\subseteq \varphi(P)$.
\end{itemize}

Then, we have the following Lemma:

\begin{lem}\label{C2L1}
Let $(A,t)$ be a  ${\cal C}_k$-algebra, $P\subseteq A$ a prime filter, and $\varphi$ is the Birula--Rasiowa  transformation on $A$. Then, the following properties hold:

\begin{enumerate}
\item[\rm (a)] $t^i(P)$ is a prime filter for  $1\leq i\leq k$;
\item[\rm (b)] $P$ is minimal {\rm(}maximal{\rm)} iff  $t^i(P)$ is minimal {\rm(}maximal{\rm)} for $1\leq i\leq k$; 
\item[\rm (c)] $U$ is an ultrafilter iff $t^i(U)$ is an ultrafilter for $1\leq i\leq k$;
\item[\rm (d)] $\varphi(t^i(P))=t^i \varphi(P)$, $1\leq i\leq k$;
\item[\rm (e)] if $F$ is a $c$-filter of $A$ and $F\subseteq P$, then $F\subseteq \varphi(P)\cap t^i(P)$;
\item[\rm (f)] if $P\subseteq Q$ and $Q$ is a prime filter of  $A$, then $\varphi(P)=Q$ \'o $P=Q$. 
\end{enumerate}
\end{lem}

\begin{dem}
It follows from the very definitions.
\end{dem}

\

In what follows, we will study the prime spectrum of a given ${\cal C}_k$-algebra $A$ with an algebraic technique. First, let us observe the study could be done using topological representation through the techniques given in \cite{AFO2,NO}.

\begin{prop} \label{C2P1}
Let $(A,t)$ be a ${\cal C}_k$-algebra and  $U\subseteq A$  an ultrafilter. Then, $N= \bigcap \limits_{i=0}^{k-1} t^i U \cap \varphi(t^i U)$ is a  $c$-filter of $A$. 
\end{prop}

\begin{dem}
It is clear that $N$ is a filter of  $A$. Then, we only have to prove that  $N$ is closed by $\bigtriangleup$ and $t$. Indeed, let $x\in N$ then by  (T14) we have that  $\sim x \vee \bigtriangleup x\in N$. Since $x\in \varphi(t^i U)$ for some $i$ such that $0,\leq i \leq k-1$, we infer that   $\sim x\not\in t^i U$ and now from  Lemma \ref{C2L1} (a) we have that   $\bigtriangleup x\in t^i U$. Now, if  $x\in t^i(U)$, then  $\sim x\not\in \varphi( t^i U)$ and so $\bigtriangleup x\in \varphi( t^i U)$. From the last assertions, we deduce that   $\bigtriangleup x \in N$. On the other hand, let us suppose $x\in N$, if $x\in \bigcap \limits_{i=0}^{k-1} t^i (U)$, then  $t(x)\in \bigcap \limits_{i=0}^{k-1} t^{i}(U)$.  Besides, if we have that $x\in \bigcap \limits_{i=0}^{k-1} \varphi(t^i U)$, then $t(x)\in \bigcap \limits_{i=0}^{k-1} t(\varphi t^{i}(U))$ and from Lemma \ref{C2L1} (d) we conclude that   $t(x)\in \bigcap \limits_{i=0}^{k-1} \varphi (t^{i} U)$. Therefore, $t(x)\in N$ as desired.
\end{dem}

\

In the following, we will show that every maximal $c$-filter can be expressed in terms of certain ultrafilters extending the result given \cite[Theorem 5.10]{AFO1}.

\begin{theo} \label{C2T1}
Let $(A,t)$ be a ${\cal C}_k$-algebra  and  $N \subseteq A$. Then,  $N$ is a maximal $c$-filter of if only if  there exists an  ultrafilter $U$ of $A$ such that  $$N=\bigcap \limits_{i=0}^{k-1} t^i U \cap \varphi(t^i U).$$
Besides, this representation of $N$ is  unique.
\end{theo}

\begin{dem}
Let $N$ be a maximal $c$-filter of $A$, since $N$ is proper there is an  ultrafilter $U$ of $A$ such that  $N\subseteq U$. Then, from  Lemma \ref{C2L1} we infer that $N\subseteq\bigcap \limits_{i=0}^{k-1} t^i U \cap \varphi(t^i U)$ and so $N=\bigcap \limits_{i=0}^{k-1} t^i U \cap \varphi(t^i U)$. 

Conversely, let us take  $U$ an ultrafilter of $A$ such that  $N = \bigcap \limits_{i=0}^{k-1} t^i U \cap \varphi(t^i U)$. From Proposition \ref{C2P1}, we have that  $N$ is a $c$-filter of $A$ and since  $N$ is a proper filter, then there is a maximal $c$-filter $W$ such that $N\subseteq W$. From the necessary condition, we know that there exists an ultrafilter $Y$ of $A$ such that $W= \bigcap \limits_{i=0}^{k-1} t^i Y \cap \varphi(t^i Y) $. Now, since  $
Y$ is a prime filter, we infer that  $t^i(U) \subseteq Y$  or $\varphi(t^i U) \subseteq Y$. Besides, from  Lemma \ref{C2L1}, we know that  $t^i(U)$ is maximal, and therefore  $t^i(U) = Y$. On the other hand, if   $\varphi(t^i U) \subseteq Y$, from Lemma \ref{C2L1}, we have that  $t^i U = \varphi(\varphi(t^i U)) = Y$ \'o $\varphi(t^i U) =Y$. Therefore, $N=W$ as desired.
\end{dem}

\

\begin{defi}\label{dfiltro}
Let $(A,t)$ be a ${\cal C}_k$-algebra, we say that a maximal $c$-filter  $N$ is $d$-periodic if  $d$ the smallest  positive integer such that  $t^d N=N$. In this case, we also say   $N$ is a $d$-filter of $A$.
\end{defi}

The last Definition plays an important role in the following Lemma.

\begin{lem}\label{div}
Let $(A,t)$ be a ${\cal C}_k$-algebra. If $N$ is a $d$-filter of  $A$, then the following conditions hold:

\begin{itemize}
\item[\rm (i)] $d/k$ where $/$ is the divide relation on integers,
\item[\rm (ii)] the unique ultrafilters of  $A$ such that they contain  $N$ are the form $\{ t^i P, \varphi(t^i P)\}_{\{ 0\leq i\leq d-1\}}$.
\end{itemize}
\end{lem}

\begin{dem}

\noindent (i): According to Definition  \ref{dfiltro}, it is clear that $d\leq k$, then we have that $k=dq+r$ with $r<d$. Thus,  $t^k N = t^{dq+r}N = t^r(t^{dq} N) =t^rN$ and since $(A,t)$ is   a ${\cal C}_k$-algebra, we infer that  $t^r N= N$ as desired.

\noindent (ii): It is direct consequence from  Definition \ref{dfiltro} and Theorem \ref{C2T1}.
\end{dem}

\section{The semsimplicity of the class of ${\cal C}_k$-algebras} 

In the following, we will characterize the congruences by means of certain deductive systems, for a given ${\cal C}_k$-algebra in order to prove that the variety of   ${\cal C}_k$-algebra is semisimple. Now, let $(A,t)$ be a ${\cal C}_k$-algebra, then we define a new implication $\rightharpoondown$ on $A$ as follows:

$$ a\rightharpoondown b = \bigvee \limits_{i=1}^{k} \nabla( \sim t^i a) \vee b.$$

Thus, we have the following Lemma that will be central in the rest of the section.

\begin{lem}\label{C2L2} In every  ${\cal C}_k$-algebra the following conditions hold:

\begin{multicols}{2}
\begin{enumerate}
  \item[\rm (i)]  $a\rightharpoondown a= 1$,
  \item[\rm (ii)] $a\rightharpoondown \triangle a = 1$,
  \item[\rm (iii)] $a\rightharpoondown t^s(a) = 1$, for every  $s\in \NN$,
  \item[\rm (iv)]  $a\rightharpoondown (a\wedge b) = a\rightharpoondown b$,
  \item[\rm (v)]  $a\leq b$ implies $c\rightharpoondown a \leq c\rightharpoondown b $,
  \item[\rm (vi)]   $((a\rightharpoondown b) \rightharpoondown  a)\rightharpoondown a)=1$,
  \item[\rm (vii)] $a\rightharpoondown ( b \rightharpoondown a)=1$,
  \item[\rm (viii)]  $1\rightharpoondown a = 1$ implies $a=1$,
  \item[\rm (ix)] $(a\rightharpoondown ( b \rightharpoondown c))= (( a\rightharpoondown b) \rightharpoondown (a\rightharpoondown c))$,
\item[\rm (x)] $a\rightharpoondown 1=1$. 
\end{enumerate}
\end{multicols}
\end{lem}

\begin{dem} We only prove  (iii), (iv) y (ix):

\noindent (iii): If $s\leq k$, then $1=\nabla \sim t^s a \vee t^s(a)\leq \bigvee \limits_{i=1}^{k} \nabla \sim t^i a \vee t^s(a)=
a\rightharpoondown t^s(a)$. On the other hand, if $k\leq s$, then there are $q,r$ positive integers such that  $s= k \cdot q + r$ and $r < k$. Thus,  $t^s(a)=t^r(a)$ and by analogous reasoning  given before we have proved the property.

\noindent (vi): Let us prove first:\, ($\ast$)$\, \bigvee\limits_{n=1}^{k-1} t^n (\bigwedge \limits_{i=1}^{k} t^i x )=  \bigvee\limits_{n=1}^{k-1} \bigwedge \limits_{i=1}^{k} t^{n+i} x = \bigwedge \limits_{j=1}^{k} t^j x $. Indeed, it is enough to show $\{ t ^{n_0+i} x\}_{1\leq i\leq k} = \{ t ^j x\}_{1\leq j\leq k}$ for every positive integer  $n_0$ such that $1\leq n_0\leq k$.

Let $A= \{ t ^j x\}_{1\leq j\leq k}$ and $B=\{ t ^{n_0+i} x\}_{1\leq i\leq k}$. If $j$ is such that  $1\leq j\leq k$, then we have three possible cases:\, (a) $j>n_0$\, or (b) $j=n_0$\, or (c) $j<n_0$. 

We will see there is an positive integer``$i$'' such that $t^j x = t^{n_0+i} x$ for every case. Indeed, if (a) occurs,   we have that   $0<j-n_0<k$, then taking  $i=j-n_0$ the property is proved. If  (b) holds, taking  $i=k$ the property is easily verified . Finally, if (c) holds, then we have that  $0<n_0-j<k$ and so $0<k-(n_0-j)<k$. Hence, if we take  $i= k-(n_0-j)$, then we have that  $t^j x = t^{n_0+i} x$.  From the last assertions, we conclude that  $A\subseteq B$.  
 
 Conversely, since $0\leq i \leq k$ we have that  $n_0\leq i+n_0 \leq k+n_0$. Then, we have the following two cases. If (a) occurs we have  $n_0\leq i+n_0 \leq k$, then  $j= i+n_0$ as desired. If (b) occurs we have $k < i + n_0 \leq k + n_0$,  then $i + n_0= q\cdot k + r$ with $0\leq r< k-1$. Thus, $t^{i + n_0}x=t^{q\cdot k + r} x = t^r x$ and if we take   $j=r$, we have the other inclusion verified. Therefore,  $A=B$ as desired.

Now, let us show  that  $((a\rightharpoondown b) \rightharpoondown  a)= a$. First, let us observe that  ($\ast\ast$) $a\rightharpoondown b = \bigvee\limits_{j=1}^{k} \nabla\sim t^j a \vee b =  \nabla\sim (\bigwedge\limits_{j=1}^{k}t^j a) \vee b = \nabla\sim z \vee b$, where $z=\bigwedge\limits_{j=1}^{k}t^j a$. Thus,

\noindent $(a\rightharpoondown b) \rightharpoondown  a =( \nabla\sim z \vee b) \rightharpoondown  a)= \bigvee\limits_{n=1}^{k} \nabla\sim t^n ( \nabla\sim z \vee b) \vee  a = \nabla\sim ( \nabla\sim z \vee b) \vee \bigvee\limits_{n=1}^{k-1} \nabla\sim t^n ( \nabla\sim z \vee b) \vee  a = \nabla ( \sim\nabla\sim z \wedge \sim b) \vee \bigvee\limits_{n=1}^{k-1}  t^n \nabla(\sim \nabla\sim z  \wedge \sim b) \vee  a = \nabla ( \triangle z \wedge \sim b) \vee \bigvee\limits_{n=1}^{k-1}  t^n \nabla(\triangle z  \wedge \sim b) \vee  a =(T17)\,\,  ( \triangle z \wedge \nabla\sim b) \vee \bigvee\limits_{n=1}^{k-1}  t^n (\triangle z  \wedge \nabla\sim b) \vee a$. Then, replacing  $z$ we have:

\noindent $(a\rightharpoondown b) \rightharpoondown  a =( \triangle \bigwedge\limits_{j=1}^{k}t^j a \wedge \nabla\sim b) \vee \bigvee\limits_{n=1}^{k-1}  t^n (\triangle \bigwedge\limits_{j=1}^{k}t^j a  \wedge \nabla\sim b) \vee  a= (  \bigwedge\limits_{j=1}^{k}t^j \triangle a \wedge \nabla\sim b) \vee \bigvee\limits_{n=1}^{k-1}  t^n ( \bigwedge\limits_{j=1}^{k}t^j \triangle a  \wedge \nabla\sim b) \vee  a = ( \triangle a \wedge \bigwedge\limits_{j=1}^{k-1}t^j \triangle a \wedge \nabla\sim b) \vee \bigvee\limits_{n=1}^{k-1}  t^n ( \bigwedge\limits_{j=1}^{k}t^j \triangle a  \wedge \nabla\sim b) \vee  a = (( \triangle a \wedge \bigwedge\limits_{j=1}^{k-1}t^j \triangle a \wedge \nabla\sim b) \vee  a) \vee \bigvee\limits_{n=1}^{k-1}  t^n ( \bigwedge\limits_{j=1}^{k}t^j \triangle a  \wedge \nabla\sim b) =(T2) \,\, a \vee \bigvee\limits_{n=1}^{k-1}  t^n ( \bigwedge\limits_{j=1}^{k}t^j \triangle a  \wedge \nabla\sim b) = 
a \vee (\bigvee\limits_{n=1}^{k-1}  t^n( \bigwedge\limits_{j=1}^{k}t^j \triangle a)  \wedge  t^n  \nabla\sim b)$. 
From the latter, ($\ast\ast$) and (T2), we can infer that

\noindent $(a\rightharpoondown b) \rightharpoondown  a = a \vee (\bigwedge \limits_{i=1}^{k} t^i \triangle a  \wedge \bigvee\limits_{n=1}^{k-1}   t^n  \nabla\sim b) = a \vee ( \triangle a \wedge \bigwedge \limits_{i=1}^{k-1} t^i \triangle a  \wedge \bigvee\limits_{n=1}^{k-1}  t^n  \nabla\sim b) = a$. Form the last assertion and (i), we have proved the property.

\noindent (ix): First, let us observe that $( a\rightharpoondown b) \rightharpoondown (a\rightharpoondown c) =  \bigvee\limits_{i=1}^{k} \nabla \sim t^i (\bigvee\limits_{j=1}^{k} \nabla \sim t^j a \vee b) \vee (\bigvee\limits_{j=1}^{k} \nabla \sim t^j a \vee c) =  \bigvee\limits_{i=1}^{k} \nabla t^i (\bigwedge\limits_{j=1}^{k} \sim \nabla \sim t^j a \wedge \sim b) \vee (\bigvee\limits_{j=1}^{k} \nabla \sim t^j a \vee c)  =  \bigvee\limits_{i=1}^{k}  t^i (\bigwedge\limits_{j=1}^{k} \triangle t^j a \wedge \sim b) \vee (\bigvee\limits_{j=1}^{k} \nabla \sim t^j a \vee c)  = (T10)\,(T17) \,\,  \bigvee\limits_{i=1}^{k}  t^i \triangle (\bigwedge\limits_{j=1}^{k}  t^j a \wedge \nabla \sim b) \vee \bigvee\limits_{j=1}^{k} \nabla \sim t^j a \vee c =(T10)\,\, (\bigvee\limits_{i=1}^{k}  t^i \bigwedge\limits_{j=1}^{k} \triangle t^j a \wedge \bigvee\limits_{i=1}^{k}  t^i \nabla \sim b) \vee (\bigvee\limits_{j=1}^{k} \nabla \sim t^j a \vee c) = (\bigwedge\limits_{j=1}^{k}  t^j \triangle a \vee \bigvee\limits_{i=1}^{k-1}  t^i \bigwedge\limits_{j=1}^{k}  t^j \triangle a) \wedge \bigvee\limits_{i=1}^{k}  t^i \nabla \sim b) \vee (\bigvee\limits_{j=1}^{k} \nabla \sim t^j a \vee c) = (\ast \ast)\, ((\bigwedge\limits_{j=1}^{k}  t^j \triangle a \vee \bigwedge\limits_{j=1}^{k}  t^j \triangle a ) \wedge \bigvee\limits_{i=1}^{k}  t^i  \nabla \sim b) \vee \bigvee\limits_{j=1}^{k} \nabla \sim t^j a \vee c = (\bigwedge\limits_{j=1}^{k}  t^j \triangle a   \wedge \bigvee\limits_{i=1}^{k}  t^i  \nabla \sim b) \vee (\bigvee\limits_{j=1}^{k} \nabla \sim t^j a \vee c) = ((\bigwedge\limits_{j=1}^{k}  t^j \triangle a \vee \bigvee\limits_{j=1}^{k} \nabla \sim t^j a)  \wedge (\bigvee\limits_{i=1}^{k}  t^i  \nabla \sim b \vee \bigvee\limits_{j=1}^{k} \nabla \sim t^j a)) \vee c  =((\bigwedge\limits_{j=1}^{k}  t^j \triangle a \vee \bigvee\limits_{j=1}^{k}  t^j \sim \triangle a)  \wedge (\bigvee\limits_{i=1}^{k}  t^i  \nabla \sim b \vee \bigvee\limits_{j=1}^{k} \nabla \sim t^j a)) \vee c$. Since by (T4) we know that  $\triangle a$ is a Boolean element, then $(\bigwedge\limits_{j=1}^{k}  t^j \triangle a \vee \bigvee\limits_{j=1}^{k}  t^j \sim \triangle a)= 1$. Thus, we have:  

$( a\rightharpoondown b) \rightharpoondown (a\rightharpoondown c) = \bigvee\limits_{i=1}^{k}  t^i  \nabla \sim b \vee \bigvee\limits_{j=1}^{k} \nabla \sim t^j a \vee c =  \bigvee\limits_{j=1}^{k} \nabla (\sim t^j a) \vee\bigvee\limits_{i=1}^{k} \nabla (\sim t^i    b) \vee  c = a\rightharpoondown (b \rightharpoondown  c)$ which completes the proof.
\end{dem} 

\

In the following, we will consider the notion of deductive system for the implication $\rightharpoondown$ as usual: given a ${\cal C}_k$-algebra $A$ and $D\subset A$, we say that  $D$ is a cyclic deductive system if  (D$_c$1)\, $1\in D$ and if  (D$_c$2)\, $x,x\rightharpoondown y\in D$, then $y\in D$.

\begin{lem}\label{C2L3}
Let $(A,t)$ be a ${\cal C}_k$-algebra and $F\subseteq A$, then the following conditions are equivalent:

\begin{itemize}
  \item[\rm (i)] $F$ is a cyclic deductive system,
  \item[\rm (ii)] $F$ is a  $c$-filter.
\end{itemize} 
\end{lem}

\begin{dem}

\noindent (i)$\Rightarrow$ (ii): It is clear that $1\in F$. Now, let us suppose  $a,b\in F$, from  Lemma \ref{C2L2} (vii) and (iv), we have $1=b\rightharpoondown (a \rightharpoondown b)= b \rightharpoondown ( a\rightharpoondown (a\wedge b)) \in F$. Then, by (D$_c$2) we infer that  $a\wedge b \in F$. Now, let us suppose  $a\in F$ and $b\in A$ such that  $a\leq b$, then form  Lemma \ref{C2L2} (v) we have that $1= a \rightharpoondown  a \leq a\rightharpoondown  b$. From the latter,  Lemma \ref{C2L2} (viii)  and (D$_c$2), we obtain that  $b \in F$. Now, from Lemma \ref{C2L2} (ii) and (iii), we have the conditions of  $c$-filter are verified.

\noindent (ii)$\Rightarrow$ (i): Let $a,a \rightharpoondown b\in F$, then $\triangle t^j a\in F$ for every $1\leq j \leq k$.  Besides, $\triangle (a \rightharpoondown b)\in F$ and therefore $\bigwedge \limits_{j=1}^{k}\triangle t^j a \wedge \triangle (a \rightharpoondown b)\in F$. On the other hand, we have that  $\bigwedge \limits_{j=1}^{k}\triangle t^j a \wedge \triangle (a \rightharpoondown b) = \bigwedge \limits_{j=1}^{k}\triangle t^j a \wedge \triangle (\bigvee \limits_{i=1}^{k}  \nabla \sim t^i a \vee b) = (T18) \bigwedge \limits_{j=1}^{k}\triangle t^j a \wedge (\bigvee \limits_{i=1}^{k} \nabla \sim  t^i a \vee \triangle  b) =  (\bigwedge \limits_{j=1}^{k}\triangle t^j a \wedge \bigvee \limits_{i=1}^{k} \sim \triangle t^i a ) \vee (\bigwedge \limits_{j=1}^{k}\triangle t^j a \wedge \triangle  b)=H$. Since  (T4) we know that  $\triangle a$ is a Boolean element, then  $(\bigwedge \limits_{j=1}^{k}\triangle t^j a \wedge \bigvee \limits_{i=1}^{k} \sim \triangle t^i a )=0$. Thus,  $H= \bigwedge \limits_{j=1}^{k}\triangle t^j a \wedge \triangle  b \leq \triangle  b  \leq b  $. Hence, from the last assertions we conclude that  $b\in F$ as desired. 
\end{dem}

\

From Lemma \ref{C2L3} we can say  the congruences are characterized by the cyclic deductive systems. Now, according to  Lemma \ref{C2L2} and results given by A. Monteiro (\cite{AM}), we have proved the following Lemma.

\begin{lem}  
For every ${\cal C}_k$-algebra  $(A,t)$, the following conditions hold:
\begin{itemize}
\item[\rm (i)] every cyclic deductive system of  $A$ is intersection of maximal  cyclic deductive systems of $A$,
\item[\rm (ii)] the intersection of all maximal  cyclic deductive systems of  $A$ is equal to  $\{1\}$.

\item[\rm (iii)] The variety of  ${\cal C}_k$-algebras is semisimple.
\end{itemize}
\end{lem}

\section{ Generating algebras of the variety} 

In this section, we will determine the simple algebras of the class of ${\cal C}_k$-algebras. To this end, we first studied the maximal congruence for a given algebra. The notion of cyclic  deductive generated by a set $H$ is defined as usual and we denoted by $D(H)$. If $H=\{a\}$ we just write  $D(a)$.

\begin{lem}\label{C2L4}
Let $(A,t)$ be a ${\cal C}_k$-algebra, $a\in A$ and $H\subseteq A$, then:
\begin{itemize}
  \item[\rm (i)] $D(a)= [\bigwedge \limits_{j=1}^{k} t^j\triangle a )$, where $[X)$ is the $c$-filter generated by $X$,
  \item[\rm (ii)] $D(H,a)=D(H\cup \{a\})=\{x\in A: a \rightharpoondown x \in D(H)\}$,
  \item[\rm (iii)] $D(a)= \{x\in A: a \rightharpoondown x =1\}$.
\end{itemize}
\end{lem}

\begin{dem}
(i): Since $\bigwedge \limits_{j=1}^{k} t^j\triangle a \leq \triangle a \leq a$, then  $a\in [\bigwedge \limits_{j=1}^{k} t^j\triangle a )$. Now, let us take   $x\in [\bigwedge \limits_{j=1}^{k} t^j\triangle a )$, then $\bigwedge \limits_{j=1}^{k} t^j\triangle a \leq x$ and  $\bigwedge \limits_{j=1}^{k} t^j\triangle \triangle  a \leq \triangle x$. Thus,  $\triangle x\in [\bigwedge \limits_{j=1}^{k} t^j\triangle a )$. Similarly, we can prove that   $tx\in [\bigwedge \limits_{j=1}^{k} t^j\triangle a )$. Then,  from Lemma \ref{C2L3} we have that  $[\bigwedge \limits_{j=1}^{k} t^j\triangle a )$ is a cyclic deductive system.  On the other hand, if  $W$ is a cyclic deductive system  and  $a\in W$, then it is simple to verify that  $[\bigwedge \limits_{j=1}^{k} t^j\triangle a )\subseteq W$ which completes the proof.

(ii): From Lemma \ref{C2L2} (i), we have that  $a\in D(H,a)$. Besides, it is clear that  $h\in D(H,a)$ for every  $h\in H$. Let us see that  $D(H,a)$ is a deductive system. Indeed, from Lemma \ref{C2L2} (x) we know that  $1\in D(H,a)$. Besides, let us consider $x, x\rightharpoondown y\in D(H,a)$, then we have that  $a\rightharpoondown x, a\rightharpoondown (x\rightharpoondown y)  \in D(H)$. Thus, from Lemma \ref{C2L2}  we infer that   $1= (a\rightharpoondown (x\rightharpoondown y)) \rightharpoondown ((a\rightharpoondown x) \rightharpoondown (a\rightharpoondown y)) \in D(H)$. From the latter and  (D$_c$2), we have  $a\rightharpoondown y \in D(H)$. 
On the other hand, if $B$ is a cyclic deductive system such that   $H\cup \{a\}\subseteq B$, then is not hard to see that   $D(H\cup \{a\})\subseteq B$ as desired.

(iii): It  immediately follows from  (ii) taking  $H = \emptyset$.
\end{dem}

\begin{lem} \label{C2L5}
Let $(A,t)$ be a ${\cal C}_k$-algebra and let  $D_1$ be a cyclic deductive system of  $ A$. Then,  $$D(D_1,a)=  \{x\in A:\, {\rm there\,\, exists } \,\, d\in D_1\, {\rm such \,\, that } \,\, d \wedge \bigwedge \limits_{j=1}^{k} t^j \triangle a\leq x \}.$$
\end{lem}

\begin{dem}
Consider $B=\{x\in A:\, {\rm there\,\, is}\, \, d\in D_1\, {\rm such\,\, that} \, d \wedge \bigwedge \limits_{j=1}^{k} t^j \triangle a\leq x \}$. Then, it is clear that  $D_1\subseteq D(D_1,a)$ and that $a\in D(D_1,a)$. On the other hand, if  $x,y \in B$ and since  $D_1$ is filter then we conclude  $x\wedge y\in B$. 
Now let us see that $B$ is a $c$-filter. Indeed, let $z\in B$ then there is $w\in D_1$ such that $w\wedge \bigwedge \limits_{j=1}^{k} t^j \triangle a \leq z$. Thus, from  (T10) we have that  $\triangle (w\wedge \bigwedge \limits_{j=1}^{k} t^j \triangle a) = \triangle w\wedge \bigwedge \limits_{j=1}^{k} t^j \triangle a  \leq \triangle z$. From  Lemma \ref{C2L3}, $D_1$ is a $c$-filter we have that  $\triangle w \in D_1$ and then $\triangle z\in B$.  Analogously, it is possible to see that  $t(z)\in B$. 
On the other hand, let us suppose there exists a $c$-filter $Y$ such that $D_1\cup \{a\}\subseteq Y$. If  $x\in B$, then there exists  $d\in D_1$ such that   $d \wedge \bigwedge \limits_{j=1}^{k} t^j \triangle a\leq x$. Now, since $\bigwedge \limits_{j=1}^{k} t^j \triangle a\in Y$ and $Y$ us a filter we infer that  $x\in Y$, which concludes the proof. 
\end{dem}


\

Our next task is to determine the simple algebras of the variety of ${\cal C}_k$-algebra. First, recall that   $A/ \theta$ is  a simple algebra if only if  $\theta$ is a maximal congruence for a given ${\cal C}_k$-algebra $A$, \cite[p.59]{SB.HS}. Then, taking into account  Remark  \ref{Ockc}  and Lemma  \ref{C2L3}, the maximal congruences are determined by the cyclic deductive system.

\begin{theo}\label{carac}
Let $(A,t)$ be a ${\cal C}_k$-algebra and $M\subseteq A$ a cyclic deductive system. Then, the following conditions are equivalents:
\begin{itemize}
  \item[\rm (1)] $M$ is maximal,
  \item[\rm (2)] if $a\not\in M$, then there exists $m\in M$ such that  $\bigwedge \limits_{j=1}^{k} t^j(\triangle a) \wedge m  = 0$, 
  \item[\rm (3)] if $\bigwedge \limits_{j=1}^{k} t^j(\triangle a) \vee b\in M$, then  $a\in M$ or $b\in M$, 
  \item[\rm (4)] if $a\not\in M$, then $\sim \triangle \bigwedge \limits_{j=1}^{k} t^j( a ) \in M$,
  \item[\rm (5)] if $a,b\not\in M$, then  $a\rightharpoondown b, b\rightharpoondown a\in M$.
\end{itemize}
\end{theo}

\begin{dem}

\noindent (1)\, $\Rightarrow$\, (2): Let $a\in A$ such that  $a\not\in M$ and consider $D=D(M,a)$. If $\bigwedge \limits_{j=1}^{k} t^j(\triangle a) \wedge m  \not= 0$ holds for every  $m\in M$, then from  Lemma \ref{C2L5} we have that  $D\not= A$ and $M\subset D\subset A$, which is impossible.

\noindent (2)\,$\Rightarrow$\, (3): Let $b\in A$ such that $\bigwedge \limits_{j=1}^{k} t^j(\triangle a)\vee b \in M$ and suppose that  $a\not\in M$. Then, by  (2) there exists $m\in M$ such that   $\bigwedge \limits_{j=1}^{k} t^j(\triangle a) \wedge m  = 0$. Thus,  $(\bigwedge \limits_{j=1}^{k} t^j(\triangle a)\vee b)\wedge m = (\bigwedge \limits_{j=1}^{k} t^j(\triangle a )\wedge m)\vee (b\wedge m)=b\wedge m \in M$ and therefore $b\in M$.

\noindent  (3)\,$\Rightarrow$\, (4): By hypothesis  $a\not\in M$ and since   $\triangle \bigwedge \limits_{j=1}^{k} t^j( a ) \leq a$, we have  $\triangle \bigwedge \limits_{j=1}^{k} t^j( a )\not\in M$. On the other hand,  $\sim \triangle \bigwedge \limits_{j=1}^{k} t^j( a ) \vee \triangle \bigwedge \limits_{j=1}^{k} t^j( a )  = 1\in M$. Hence, from the last assertions and  (3) we can infer $\sim \triangle \bigwedge \limits_{j=1}^{k} t^j( a )\in M$, which completes the proof.

\noindent (4)\,$\Rightarrow$\, (5): Let $a, b\in A$ such that $a\not\in M$ and $b \not\in M$. Thus, from (4)  we have that $\sim \triangle \bigwedge \limits_{j=1}^{k} t^j( a ) \in M$. Since 
$\sim \triangle \bigwedge \limits_{j=1}^{k} t^j( a )= \sim \bigwedge \limits_{j=1}^{k}   \triangle  t^j( a ) = \bigvee \limits_{j=1}^{k} \sim  \triangle  t^j( a ) = \bigvee \limits_{j=1}^{k}  \nabla \sim  t^j( a ) \leq a\rightharpoondown b$, then  $a\rightharpoondown b \in M$. Analogously, it is possible to see that  $b\rightharpoondown a\in M$.

\noindent (5)\,$\Rightarrow$\,(1): Let us suppose that  $M$ is not  maximal, then there is cyclic deductive system $M^\prime$ such that $M\subset M^\prime$. Now, let $a\in M^\prime\backslash M$ and $b\in A\backslash M^\prime$, thus by  (5) we have  $a\rightharpoondown b \in M^\prime$ but from the definition of deductive system we obtain  $b \in M^\prime$, which is a contradiction. 
\end{dem}

\

In the following we will consider  the set  $K(A)=\{x\in A: tx=x=\nabla x\}$. Now, let us present the following result essential in the rest of the section.

\begin{prop}\label{C3Prop}
Let $(A,t)$ be a ${\cal C}_k$-algebra and $a\in A$. Then, the following conditions hold:
\begin{itemize}
  \item[\rm (a)] $K(A)$ is an $Mpm$-subalgebra de $A$,
  \item[\rm (b)]  $K(A)$ is a Boolean algebra, 
  \item[\rm (c)] If $(S,t_s)$ is a  ${\cal C}_k$-subalgebra of $(A,t)$, then   $K(S) = K(A)\cap S$,
  \item[\rm (d)] $[a)$ is a $c$-filter if only if  $a\in K(A)$,
  \item[\rm (e)] if $a\in K(A)$, then $D(a)=[a)$.
\end{itemize}
\end{prop}

\begin{dem}
The proof of (a) and (b) is analogous to the one of  \cite[Proposition 5.1]{AFO1}. Besides,  (c) immediately follows from the very definitions. Moreover,  (d) has an analogous proof of  \cite[Proposition 5.2]{AFO1}. Finally,  if $a\in K(A)$, then by  (d) we have that  $[a)$ is a cyclic deductive system. Besides, if $D_1$ is a cyclic deductive system such that  $a\in D_1$, then it is immediately that 
$\bigwedge \limits_{j=1}^{k} t^j\triangle \in D_1$. Therefore, from  Lemma \ref{C2L4} we conclude that  $D(a)\subseteq D_1$ which completes the proof. 
\end{dem}

\begin{theo}\label{cap3simple}
Let $(A,t)$ be a ${\cal C}_k$-algebra, then the following conditions are equivalent:
\begin{itemize}
  \item[\rm (i)] $(A,t)$ is  simple algebra,
  \item[\rm (ii)] for every  $a\in A$, $a\not=1$, $\bigwedge \limits_{j=1}^{k} t^j(\triangle a) = 0$, 
  \item[\rm (iii)]  $K(A)=\{0,1\}$.
\end{itemize}
\end{theo}
\begin{dem}

\noindent (i) $\Leftrightarrow$ (iii): It is analogous to the one of  \cite[Theorem 5.3]{AFO1}.

\noindent (i) $\Rightarrow$ (ii): Since $(A,t)$ is simple algebra, then $\{1\}$ is a maximal cyclic deductive system. From the last assertions and  Theorem \ref{carac} (2), we infer that  $\bigwedge \limits_{j=1}^{k} t^j(\triangle a) = 0$.

\noindent (ii) $\Rightarrow$ (iii): Let $z\in K(A)$ and $z\not=1$. Since $\triangle z = z = t^jz$ we have that   $t^j\triangle z= t^jz$  and therefore  $\bigwedge \limits_{j=1}^{k} t^j(\triangle z) = z= 0$ as desired.
\end{dem}

\

\begin{definition} For a given ${\cal C}_k$-algebra $(A,t)$, we say that $(A,t)$ is $r$-periodic if  $r$ is the smallest non-negative  element such that  $t^r (x)=x$ for every  $x\in A$.
\end{definition}

Let $T_2$, $T_3$, and  $T_4$ be the simple algebras of the variety of  $Mpm$-\'algebras. In the following, we will consider the sets $T_{2,k}$, $T_{3,k}$, and $T_{4,k}$ of all sequences $x=(x_1,\cdots, x_k)$ with $x_i\in T_{s,k}$ ($s=2,3,4$) and with the pointwise defined operations. Indeed, these algebras are also ${\cal C}_k$-algebras. Taking the function $t:T_{i,k}\to T_{i,k}$ defined as follows  $t(x_1,x_2,\cdots, x_k)=(x_k,x_1,x_2,\cdots, x_{k-1})$ where $(x_1,x_2,\cdots, x_k)\in T_{i,k}$ with $i=2,3,4$. It is not hard to see that $(T_{2,k},t)$, $(T_{3,k},t)$ and  $(T_{4,k},t)$ are $k$-periodic ${\cal C}_k$-algebras ;and, in the next, we will show they are simple algebras.

\begin{coro}\label{azg}
Every  ${\cal C}_k$-algebras $(T_{i,k},t)$ with $i=2,3,4$ is simple one.
\end{coro}

\begin{dem} In the first place consider $(T_{4,k},t)$. By construction, we have that  $0,1\in K(T_{4,k})$. Let us show that they are the unique elements of  $K(T_{4,k})$. Indeed, let $z\in K(T_{4,k})$ such that  $z\not= 0$ and $z\not= 1$, then  $z=(x_1,x_2,\cdots, x_k)$. Thus, we have the following three cases: (I) $x_i\in \{a,b\}$ for some $i$, $1\leq i\leq k$ or  (II) $x_i\in \{0,1\}$ for every $i$, $1\leq i\leq k$. If we have (I), since  $\nabla a= \nabla b= 1$, we infer that  $\nabla z \not = z$, which is a contradiction. If  (II) holds, from the hypothesis there exists $l$ such that  $x_l=0$ and $x_{l+1}=1$ or $x_l=1$ and $x_{l+1}=0$. In any case, from definition of  $t$ we have that  $tz\not=z$, which contradicts that $z\in K(T_{4,k})$. On the other hand, from Theorem  \ref{cap3simple} we have that  $(T_{4,k},t)$ is a simple ${\cal C}_k$-algebra. Analogously, it is possible to see that $(T_{2,k},t)$ and $(T_{3,k},t)$ are simple ${\cal C}_k$-algebras.
\end{dem}

\begin{coro}
The subalgebras $(T_{i,k},t)$ with $i=3,4$ are simple algebras. 
\end{coro}
\begin{dem}
It  immediately follows from Proposition \ref{C3Prop}.
\end{dem}

\begin{lem}\label{simpfinit}
Every simple  ${\cal C}_k$-algebra is finite.
\end{lem}

\begin{dem}
Let $A$ be a  simple algebra of the variety. Then,  $\{1\}$ is a maximal $c$-filter of $A$ and from   Theorem \ref{C2T1} there is an ultrafilter $U$ such that  $\{1\}=\bigcap \limits_{i=0}^{k-1} t^i U \cap \varphi(t^i U)$. From the latter, we can infer that the unique  ultrafilters of $A$ are $\{ t^i U ,\varphi(t^i U)\}_{\{0\leq i\leq k-1\}}$, which is a finite set. Thus, the prime spectrum of  $A$ is finite and so  $A$ is finite. 
\end{dem}

\begin{theo}\label{subsimples}
If $(A,t)$ is a simple ${\cal C}_k$-algebra, then   $(A,t)\simeq (T_{i,d},t)$ with $i=2,3,4$ and $T_{i,d}$ is $d$-periodic for some positive integer  $d$ such that $d/k$.
\end{theo}

\begin{dem}
According to Lemma \ref{simpfinit} we have that  $A$ is finite and let us suppose that $\{1\}$ is $d$-periodic with $d/k$. Then, $\{1\}=\bigcap \limits_{i=0}^{d-1} t^i P \cap \varphi(t^i P)$ and so  the prime spectrum of  $A$ is $\{ t^j P ,\varphi(t^j P)\}_{\{0\leq j\leq d-1\}}$, see Lemma \ref{div}.

On the other hand, since $A$ is a finite lattice then there a prime element $p\in \Pi(A)$ such that  $P=[p)$. That is to say, we can consider the transformation as follows: $\psi (p) = q$ if only if  $\varphi([p))=[q)$. This allows us to identify every prime element with of the prime filter of the spectrum. Thus,  $p$ verifies the following cases: type I: $p=\psi(p)$, type II: $p<\psi( p)$ or type  III: $p$ and $\psi (p)$ are  incomparable. Now, let us suppose  $p$ is type  I, then the prime spectrum formed by the anti-chain  $\{p,tp,\ldots, t^{d-1}p\}$. It is not hard to see that  every  $t^{i}p$ is  type  I. Therefore, there is a function  $\alpha:A\to T_2^d$ where  $\alpha(p)=(1,0,\ldots,0)$, $\alpha(tp)=(0,1,0,\ldots,0)$, $\ldots$  ,$\alpha(t^jp) =(0,0,\ldots,0,1_j,0,\ldots,0)$. On the other hand, we have that $t\alpha(p)=t(1,0,\ldots,0)=(0,1,0,\ldots,0)$ and in general, we obtain  that $t^j\alpha(p)=t^{j-1}t\alpha(p) =t^{j-1}t(1,0,\ldots,0)=t^{j-1}(0,1,0,\ldots,0)=\cdots =(0,0,\ldots,0,1_j,0,\ldots,0)$. From the latter, it is not hard to see  that  $\alpha$ is  a  ${\cal C}_k$-isomorphism and therefore,   $(A,t)\simeq (T_{2,d},t)$. Now, if $p$ is  type  II, the prime spectrum is formed by  $$\{p,tp,\ldots, t^{d-1}p,\psi(p), \psi(tp),\ldots, \psi(t^{d-1}p) \}$$ where $t^i p < \psi(t^{i}p)$. Let us observe the prime filters of  $T_{3,d}$ are of the form  $(0,\ldots,0, c,0, \ldots,0)$ and $(0,\ldots,0, 1,0, \ldots,0)$. Thus, by  analogous reasoning from the last case, we have that  $(A,t)\simeq (T_{3,d},t)$. If $p$ is  type  III, the prime spectrum is the anti-chain   $$\{p,tp,\ldots, t^{d-1}p,\psi(p), \psi(tp),\ldots, \psi(t^{d-1}p) \}.$$ The prime elements of  $T_{4,d}$ are the form of  $(0,\ldots,0, a,0, \ldots,0)$ and $(0,\ldots,0, b,0, \ldots,0)$. Therefore, by  analogous reasoning we can see that $(A,t)\simeq (T_{4,d},t)$, which completes the proof.
\end{dem}

\

It is well-known that divides relation ``$/$'' between  integers is a partial order. Then, for a given positive integer $k$,  we have the set $Div(k)=\{z: z\,\, \text{divisor of}\,\, k\}$ can be considered as distributive latices; moreover, $k_1\wedge k_2=lcm(k_1,K_2)$ and $k_1\vee k_2=gcd(k_1,K_2)$ are the {\em least common multiple} and {\em greatest common divisor}, respectively. Now, it is possible to see that for every $k$-periodic Boolean algebra, we have ${\cal B}_k$ is lattice-isomorphic to the Boolean algebra   $Div(k)$. Furthermore, for every $d\in Div(k)$ there is a unique  ${\cal B}_g$ associated with $d$ which is a subalgebra of ${\cal B}_k$. Besides, ${\cal B}_g$ is $d$-periodic characterized by $B_d=\{g\in B_k: t\,\, \text{is an automorphism}\,\, \text{and }\,\, t^d g = g \}$.

\begin{lem}\label{sub} 
The subalgebras of $(T_{i,k},t)$ are of the form  $(T_{i,d},t_d)$  with $d/k$ and $i=2,3,4$. Besides, $T_{2,d}$ is a  ${\cal C}_k$-subalgebra of $T_{3,d}$ and $T_{4,d}$, but $T_{3,d}$ is not subalgebra of $T_{4,d}$.

\end{lem}
\begin{dem}
Let us take  $T_{i,k}^d=\{x\in T_i^k: t^dx=x\}$, then it is clear that $(T_{i,k}^d,t^\prime_d)$ is a subalgebra of  $(T_{i,k},t)$ where the restriction  $t|_{T_{i,k}^d} = t^\prime_d$ is an automorphism over $T_{i,k}^d$. On the other hand, from Lemma \ref{div} we have that the prime spectrum of $T_{i,k}^d$ has the form  $\{ t^i P, \varphi(t^i P)\}_{\{ 0\leq i\leq d-1\} }$. So, $(T_{i,k}^d,t^\prime_d)\simeq (T_{i,d},t_d)$.

On the other hand, it is clear that  $(B(T_{i,k}),t_B)\simeq B_k$ for $i=2,3$ where $t|_{B(T_{i,k})}=t_B$ and  $B(A)$ is the set of Boolean elements of ${\cal C}_k$-alegbra $A$. Thus, there is $d$ such that $d/k$ and  $B_d$ is a sublagebra of  $B_k$. So, since $(B(T_{i,d}),t_B)\simeq B_d$, we have that the sublagebras  $(T_{i,k},t)$ with  $i=2,3$ are the form   $(T_{i,d},t_d)$.

For the case $i=4$, it is clear that  $B(T_{4,k})=T_{4,k}$. Let us suppose that $d$ is a divisor of $k$, then $k=d\cdot q$. Now, we can consider the elements of $T_{4,k}$ as a sequence  $x=({\overline x_{1} },{\overline x_{2}},\ldots,{\overline x_{q}})$ where  ${\overline x_{1}}=(x_1, \ldots, x_q)$, $\ldots$, ${\overline x_{q}}=(x_{dq-q},x_{dq-(q-1)},\ldots, x_{dq})$. Now, let us consider the set $D=\{x\in T_{4,k}: x=({\overline x_{1}},{\overline x_{2}},\ldots,{\overline x_{q}}) , \, {\rm such\,\, that  } \,  \, \overline{x_{1}} =\overline{ x_{2}} = \ldots = \overline{ x_{d}} \} $. Then, it is clear that $D$ is a  subalgebra of  $T_{4,k}$ and taking the automorphism $t|_D:D\to D$, we have $(D,t|_D)$ is a  $d$-periodic. Their atoms are the form  $a_1=(a,0,\ldots,0), \ldots, a_d=(0,0,\ldots,a)$ and $b_1=(b,0,\ldots,0), \ldots, b_d=(0,0,\ldots,b)$ and it is verified $t(a_i)= a_{i+1}$ and $t(b_i)= b_{i+1}$. Thus, $(D,t|_D)\simeq (T_{4,d},t_d)$.

To end the proof, we will determine all sublagebras of $T_{4,k}$ where the elements of the sequences belong to  $T_2$. Let us take 
$$D^\prime=\{x\in T_{4,k}: x=({\overline x_{1}},{\overline x_{2}},\ldots,{\overline x_{q}}) , \, {\rm such \,\, that  } \,  \, \overline{x_{1}} =\overline{ x_{2}} = \ldots = \overline{ x_{d}}\in T_{2,q}\}$$ 

and  
$$t^\prime=t|_{D^\prime}:D^\prime\to D^\prime,$$ then it is not hard to see that   $(D^\prime,t^\prime)\simeq (T_{2,d},t_d)$. Therefore, $(T_{2,d},t_d)$ and $(T_{4,d},t_d)$ are the unique subalgebras of  $(T_{4,k},t)$. Since $T_3$ is not subalgebra of  $T_4$, then  $T_{3,d}$ is not subalgebra of $T_{4,d}$, which completes the proof.

\end{dem}

The following Lemma will be used in the next Section and it is an immediate consequence of Lemma \ref{sub}.


\begin{lem}
 $T_{i,d_1}\cap T_{i,d_2} = T_{i, gcd(d_1,d_2)}$,  $T_{2,d_1}\cap T_{i,d_2} = T_{2, gcd(d_1,d_2)}$, $T_{3,d_1}\cap T_{4,d_2} = T_{2, gcd(d_1,d_2)}$ with $i=2,3,4$
\end{lem}

As directly consequence from Corollary \ref{azg}, Lemma \ref{simpfinit}, Theorem \ref{subsimples} and universal algebra results given in  \cite[Theorem 10.16]{SB.HS}, we conclude: 

\begin{theo}\label{cklocfini}
The variety of  ${\cal C}_k$-algebras is finitely generated and locally finite. Besides, the only simple algebras are $(T_{i,k},t)$ with $i=3,4$ and their subalgebras.
\end{theo}

\section{ The cardinality of the finitely generated  ${\cal C}_k$-algebras}

In this section, we will focus on the task of studying the structure of free  ${\cal C}_k$-algebras with a finite number of generators. We denote by $F_{{\cal C}_k}(n)$ this algebra where $n$ is a positive integer. 
From the results showed in the last section, we know that for every $d$ divisor of $k$, the family ${\cal C}$ of the maximal $c$-filters of  $F_{{\cal C}_k}(n)$ can be partitioned in the following way:

$$N_{i,d}=\{ N\in {\cal C}:  F_{{\cal C}_k}(n)/N\simeq T_{i,d}\}.$$  

Thus,  form Theorem \ref{cklocfini} we have:
$$F_{{\cal C}_k}(n)\simeq \prod \limits_{d/k} T_{2,d}^{\alpha_{2,d}} \times \prod \limits_{d/k} T_{3,d}^{\alpha_{3,d}} \times \prod \limits_{d/k} T_{4,d}^{\alpha_{4,d}},$$
where $\alpha_{i,d}=|N_{i,d}|$ with $i=2,3,4$.


\vspace{2mm}

Now, let us consider $Epi(F_{{\cal C}_k}(n),T_{i,d})$  the set of all epimorphisms from $F_{{\cal C}_k}(n)$ into $T_{i,d}$ and we denote by $Aut(T_{i,d})$ the set of all automorphisms  over the algebra  $T_{i,d}$ and  $|X|$ denotes  the number of elements of $X$. It is not hard to see that $Aut(T_{l,d})$ (with $l=2,3$) has only $d$ automorphisms and they are   $t,t^2, \cdots, t^{d-1},t^d$. Besides, it is possible to prove that $|Aut(T_{4,d})|=2d$. Taking the function $s: Epi(F_{{\cal C}_k}(n),T_{i,d})\to N_{i,d} $ defined by $s(h)=ker(h) = h^{-1}(\{1\})$ for every  $h\in Epi(F_{{\cal C}_k}(n),T_{i,d})$, we can see that $s$ is onto and $s^{-1}(N)= \{\alpha\circ h: \alpha \in Aut(T_{i,d})\}$. Therefore, 

$$ \alpha_{i,d} = \frac{|Epi(F_{{\cal C}_k}(n),T_{i,d})|}{|Aut(T_{l,d})|}$$

\

\


On the other hand, for every  $h\in Epi(F_{{\cal C}_k}(n),T_{i,d})$ there is a function  $f:G \to T_{i,d}$ such that  $f=h|_G$. Now, if   $F^\ast(G,T_{i,d})=\{f:G \to T_{i,d}$ such that  $[f(G)]_{{\cal C}_k} = T_{i,d}\}$, then we have: 

$$ |Epi(F_{{\cal C}_k}(n),T_{i,d})|= |F^\ast(G,T_{i,d})|.$$

Where $[f(G)]_{{\cal C}_k}$ denotes the ${\cal C}_k$-algebra generated by $f(G)$. Let us observe that the condition   $[f(G)]_{{\cal C}_k} = T_{i,d}$ is equivalent to asking  $f(G)\subseteq T_{i,d}$ and $f(G)\not\subseteq S$ for every maximal subalgebra $S$ of $T_{i,d}$. If we denote $M(d)$ the set of maximal divisors of $d$ different of $d$. Then, for every maximal subalgebra of  $T_{i,d}$ are of the form $T_{i,x}$ with $x\in M(d)$. Thus, we can write:

$$F^\ast(G,T_{i,d})= F_{i,d} \backslash \bigcup\limits_{l \leq i}\, \bigcup\limits_{x\in M(d)}F_{l,x},$$

where   $F_{i,d}$ is the set of all functions from  $G$ into $T_{i,d}$. Therefore, 

$$ |F^\ast(G,T_{i,d})|= |F_{i,d}|- |\bigcup\limits_{l \leq i}\, \bigcup\limits_{x\in M(d)}F_{l,x}|= (i^d)^n - |\bigcup\limits_{l\leq i }\, \bigcup\limits_{x\in M(d)}F_{l,x}|.$$

On the other hand, it is well-known that every finite set  ${\cal J}$ and the family of finite sets $\{A_j\}_{j\in {\cal J}}$ verify:  
$$ |\bigcup\limits_{j\in {\cal J}} A_i| = \sum\limits_{X\subseteq {\cal J},\, X\not=\emptyset } (-1)^{|X|-1} | \bigcap\limits_{j\in X} A_j|.$$
Now, let us observe that if we take  $I=\{w: w\leq i\}$, then: 
$$|\bigcup\limits_{l\leq i}\bigcup\limits_{x\in M(d)}\,F_{l,x}|= |\bigcup\limits_{(z,t)\in I\times M(d)}\, F_{z,t}| = \sum\limits_{X\subseteq I\times M(d)} (-1)^{|X|} \, |\bigcap\limits_{(j_1,j_2)\in X}F_{j_1,j_2}|,$$ 
where 
$$\bigcap\limits_{(j_1,j_2)\in X} F_{j_1,j_2} = \{f\in F_{i,d}:\, \, f:G\to \bigcap\limits_{(j_1,j_2)\in X} T_{j_1,j_2}\}.$$

\section*{Calculating $\mbox{\boldmath $\alpha_{i,d}$}$}

If $i=2$, then  $I=\{2\}$ and therefore the following holds: 

$$ \alpha_{2,d}=  \frac{ (2^d)^n -|\bigcup\limits_{x\in M(d)} F_{2,x}|}{d}= \frac{ (2^d)^n -\sum\limits_{Z\subseteq  M(d), Z\not=\emptyset } (-1)^{|Z|-1} \, |\bigcap\limits_{w \in Z}F_{2,w}|}{d}$$ $$= \frac{ (2^d)^n -\sum\limits_{Z\subseteq  M(d), Z\not=\emptyset } (-1)^{|Z|-1} \, (2^{mcd(Z)})^n}{d},$$
where $gcd(Z)$ is the set of greatest common divisors of $Z$. 

\vspace{2mm}

If $i=3$, then:

$$ \alpha_{3,d}= \frac{ (3^d)^n - |\bigcup\limits_{(j,x)\in \{2,3\}\times M(d)} F_{j,x}\cup F_{2,d} |}{d} = \frac{ (3^d)^n -|F_{2,d} \cup \bigcup\limits_{y\in  M(d)} F_{3,y}|}{d}.$$


\noindent Since
$$|F_{2,d} \cup \bigcup\limits_{y\in  M(d)} F_{3,y}|= | F_{2,d}| + |\bigcup\limits_{y\in  M(d)} F_{3,y}|- | F_{2,d} \cap \bigcup\limits_{y\in  M(d)} F_{3,y}|$$
and 
$$| F_{2,d} \cap \bigcup\limits_{y\in  M(d)} F_{3,y}|= | \bigcup\limits_{y\in  M(d)} ( F_{2,d} \cap F_{3,y})|,$$
we have that 
$$\alpha_{3,d} = \frac{ (3^d)^n - (2^d)^n  -  \sum\limits_{W\subseteq  M(d), W\not=\emptyset } (-1)^{|W|-1} \, (3^{gcd(W)})^n  +  \sum\limits_{H\subseteq  M(d)\times M(d), H\not=\emptyset } (-1)^{|H|-1} \, (2^{gcd(H_1\cup H_2)})^n }{d}$$
where $H_1= \{x: (x,y)\in H\}$ and $H_2= \{y: (x,y)\in H \}$. 

\vspace{2mm}

Finally, for  $i=4$ and by Lemma \ref{sub}, we have:
$$ \alpha_{4,d}=  \frac{ (4^d)^n -|\bigcup\limits_{(j,x)\in \{2,4\}\times M(d)} F_{j,x}\cup F_{2,d} |}{2d} = \frac{ (4^d)^n -|F_{2,d} \cup \bigcup\limits_{y\in  M(d)} F_{4,y}|}{2d}. $$
By analogous reasoning to the last case, we have:
 
$$\alpha_{4,d} = \frac{ (4^d)^n - (2^{d})^n  -  \sum\limits_{W\subseteq  M(d), W\not=\emptyset } (-1)^{|W|-1} \, (4^{gcd(W)})^n  +  \sum\limits_{H\subseteq  M(d)\times M(d), H\not=\emptyset } (-1)^{|H|-1} \, (2^{gcd(H_1\cup H_2)})^n }{2d}.$$

\

From the last assertions, we have proved the following Theorem:

\begin{theo}\label{tfal}
Let $F_{{\cal C}_k}(n)$ be the free ${\cal C}_k$-algebra with $n$ generators, then: 

$$|F_{{\cal C}_k}(n)|= \prod \limits_{d/k} 2^{\alpha_{2,d}} \cdot \prod \limits_{d/k} 3^{\alpha_{3,d}} \cdot \prod \limits_{d/k} 4^{\alpha_{4,d}}$$
with
\begin{itemize}
  \item[]  $\alpha_{2,d}=   \frac{ (2^d)^n -\sum\limits_{Z\subseteq  M(d), Z\not=\emptyset } (-1)^{|Z|-1} \, (2^{gcd(Z)})^n}{d}$,
  \item[]  $\alpha_{3,d} = \frac{ (3^d)^n - (2^d)^n  -  \sum\limits_{W\subseteq  M(d), W\not=\emptyset } (-1)^{|W|-1} \, (3^{gcd(W)})^n  +  \sum\limits_{H\subseteq  M(d)\times M(d), H\not=\emptyset } (-1)^{|H|-1} \, (2^{gcd(H_1\cup H_2)})^n }{d}$,
  \item[]   $\alpha_{4,d} = \frac{ (4^d)^n -  \sum\limits_{W\subseteq  M(d), W\not=\emptyset } (-1)^{|W|-1} \, (4^{gcd(W)})^n  - (2^{d})^n +  \sum\limits_{H\subseteq  M(d)\times M(d), H\not=\emptyset } (-1)^{|H|-1} \, (2^{gcd(H_1\cup H_2)})^n }{2d}$.
\end{itemize}
\end{theo}

\subsection{ Particular cases}

In this subsection, we will show that  Theorem \ref{tfal} has  particular  cases that was previously  obtained in the literature in \cite{NO1,AFO1}.

\subsubsection*{ $\mbox{\boldmath ${\cal C}_1$}$-algebras }

If $k=1$, the class of  ${\cal C}_1$-algebras coincides with the one of  $mpM$-algebras.  According to Theorem \ref{tfal}, we have that the free ${\cal C}_1$-algebra with $r$ generators is: 

$$F_{{\cal C}_1}(r) \simeq  T_{2,1}^{\alpha_{2,1}} \times  T_{3,1}^{\alpha_{3,1}} \times  T_{4,1}^{\alpha_{4,1}}.$$
Since $M(1)=\emptyset$ we have that  $\sum\limits_{Z\subseteq  M(1) } (-1)^{|Z|}\, (2^{gcd(Z)})^r=0$ and therefore  $\alpha_{2,1}=2^r$. Besides, it is possible to see that   

$$\sum\limits_{H\subseteq  M(1)\times M(1), H\not=\emptyset  } (-1)^{|H|} \, (2^{gcd(H_1\cup H_2)})^r =0, $$ 
$$ \sum\limits_{W\subseteq  M(d), W\not=\emptyset } (-1)^{|W|-1} \, (4^{gcd(W)})^r =0 = \sum\limits_{W\subseteq  M(d), W\not=\emptyset } (-1)^{|W|-1} \, (3^{gcd(W)})^r,$$
and thus, $\alpha_{3,1}=3^r-2^r$ and $\alpha_{4,1}=\frac{4^r-2^r}{2} = 2^{r-1}(2^r-1)$. Therefore,

$$ |F_{{\cal C}_1}(n)| =  2^{2^r} \times  3^{3^r-2^r} \times  4^{2^{r-1}(2^r-1)}.$$

This number was presented in  \cite[Theorem 4.5]{NO1}.

\subsubsection*{ $\mbox{\boldmath ${\cal C}_2$}$-algebras}

For the $k=2$, we have the class of ${\cal C}_2$-algebras coincides with the one of $\cal S$-algebras introduced in \cite{AFO1}. Then, the free ${\cal C}_2$-algebra with $r$ generators is:

$$F_{{\cal C}_2}(r)\simeq \prod \limits_{d=1}^{2} T_{2,d}^{\alpha_{2,d}} \times \prod \limits_{d=1}^{2} T_{3,d}^{\alpha_{3,d}} \times \prod \limits_{d=1}^{2} T_{4,d}^{\alpha_{4,d}}.$$

Since $|T_{i,d}|=i^d$, $\alpha_{2,2}=\frac{4^r-2^r}{2}$, $\alpha_{3,2}=\frac{9^r-3^r-4^r+2^r}{2}$ and $\alpha_{4,2}=\frac{4^{2r}-4^r-2^{2r}+2^r}{4}$, then 

$$|F_{{\cal C}_2}(r)|=  |T_{2,1}|^{2^r} \times |T_{2,2}|^\frac{4^r-2^r}{2}\times |T_{3,1}|^{3^r-2^r} \times |T_{3,2}|^\frac{3^{2r}-3^r-2^{2r}+2^r}{2} \times|T_{4,1}|^{\frac{4^r-2^r}{2}} \times |T_{4,2}|^\frac{4^{2r}-4^r-2^{2r}+2^r}{4}$$
\begin{eqnarray*}
&=& 2^{2^r} \times 4^\frac{4^r-2^r}{2} \times  3^{3^r-2^r} \times 9^\frac{3^{2r}-3^r-2^{2r}+2^r}{2} \times 4^{\frac{4^r-2^r}{2}}\times 16^\frac{4^{2r}-4^r-2^{2r}+2^r}{4}\\
&=& 2^{2^r} \times 4^{2\cdot\frac{4^r-2^r}{2}} \times 3^{3^r-2^r} \times 9^\frac{9^{r}-3^r-4^{r}+2^r}{2}\times 16^\frac{16^{r}-4^r-4^{r}+2^r}{4}\\
&=& 2^{2^r} \times 4^{2\cdot\frac{4^r-2^r}{2}} \times 3^{3^r-2^r} \times 9^\frac{9^{r}-3^r-4^{r}+2^r}{2}\times 16^\frac{16^{r}-3\cdot4^r+2\cdot2^r+ 4^r-2^r}{4}\\
&=& 2^{2^r} \times 4^{2\cdot\frac{4^r-2^r}{2}} \times  3^{3^r-2^r} \times 9^\frac{9^{r}-3^r-4^{r}+2^r}{2}\times 16^\frac{16^{r}-3\cdot4^r+2\cdot2^r}{4} \times 16^\frac{ 4^r-2^r}{4}\\
&=& 2^{2^r} \times 4^{3\cdot\frac{4^r-2^r}{2}} \times  3^{3^r-2^r} \times 9^\frac{9^{r}-3^r-4^{r}+2^r}{2}  \times 16^\frac{16^{r}-3\cdot4^r+2\cdot2^r}{4}.\\
\end{eqnarray*}

\vspace{-.5cm}
This number was obtained in \cite[Theorem 6.4]{AFO1}. 

\subsubsection*{ $\mbox{\boldmath ${\cal C}_k$}$-algebras with  $\mbox{\boldmath $k$}$ is prime number} 

\

$$|F_{{\cal C}_k}(r)| = \prod \limits_{N \in N_{2,d}, d/k} |T_{2,d}|^{|N_{2,d}|} \times \prod \limits_{N \in N_{3,d}, d/k} |T_{3,d}|^{|N_{3,d}|} \times \prod \limits_{N \in N_{4,d}, d/k} |T_{4,d}|^{|N_{4,d}|}  $$

$$ = \prod \limits_{ d/k} (2^d)^{\alpha_{2,d}} \times \prod \limits_{ d/k} (3^d)^{\alpha_{3,d}} \times \prod \limits_{d/k} (4^d)^{\alpha_{4,d}}  $$

$$ =  2^{\alpha_{2,1}} \times  (2^k)^{\alpha_{2,k}}  \times 3^{\alpha_{3,1}} \times  (3^k)^{\alpha_{3,k}} \times 4^{\alpha_{4,1}} \times  (4^k)^{\alpha_{4,k}}  $$

$$ =  2^{2^r} \times  (2^k)^{\frac{2^{kr}-2^r}{k} }  \times 3^{3^r-2^r} \times  (3^k)^{\frac{3^{kr}-3^r-2^{kr}+2^r}{k} } \times 4^{\frac{4^r-2^r}{2}} \times  (4^k)^{\frac{4^{kr}-4^r-2^{kr}+2^r}{2k}}  $$

\section{ Degree-preserving logic associated with the class ${\cal C}_k$-algebras}

In this section, we show that the class of ${\cal C}_k$-algebras can be seen as  suitable algebraic semantics for a logic system. There are many ways to associate a logic system with an ordered algebraic structure. For instance, it is possible to define the {\em logic that preserves degrees of truth} associated with the class of ${\cal C}_k$-algebras. This notion was introduced and studied mainly in the works \cite{bou-etal,L2,L3,L4} with the purpose of taking advantage of the multiple values of truth, more that the two classical values of truth, in opposition to the logics that only preserve truth. These logics preserve lower bounds of  truth values instead of just preserving  value $1$.

From now on, we shall denote by $\mathfrak{Fm}=\langle Fm, \wedge, \vee, \sim, ^\ast, t, \top, \bot \rangle$ the absolutely free algebra of
type (2,2,1,1,1,0) generated by some denumerable set of variables. We call $Fm$ the sentential formulas, and we shall refer to them
by lowercase greek letters $\alpha, \beta, \gamma, \dots$ and so on; and we shall denote the sets of formulas by uppercase greek
letters $\Gamma, \Delta,$ etc..\\[2mm] 
Let $\mathbb{L}_{k}^{\leq}=\langle Fm, \models_k^{\leq}\rangle$, for every $\Gamma\cup\{\alpha\}\subseteq Fm$ 

\begin{itemize}
\item[(i)] If $\Gamma=\{\alpha_1,\dots, \alpha_n\}\not=\emptyset$,

\begin{center}
\begin{tabular}{ccc}
$\alpha_1,\dots, \alpha_n \models_k^{\leq} \alpha$ & $\Longleftrightarrow$ &  $\forall (A,t)\in \mathbb{C}_k$, $\forall v\in Hom_{\mathbb{C}_k}(Fm,(A,t))$, $\forall a\in A$\\
& & if $v(\alpha_i)\geq a$, for all $i\leq n$, then $v(\alpha)\geq a$ \\
\end{tabular}
\end{center}

\item[(ii)] $\emptyset \models_k^{\leq} \alpha$ \,  $\Longleftrightarrow$  \, $\forall A\in \mathbb{C}_k$, $\forall v\in Hom_{\mathbb{C}_k}(Fm,A)$, $v(\alpha)=1$.

\item[(iii)] If $\Gamma\subseteq Fm$ is infinity,

\begin{center}
\begin{tabular}{ccc}
$\Gamma \models_k^{\leq} \alpha$ & $\Longleftrightarrow$ &  there are $\alpha_1,\dots, \alpha_n \in \Gamma$ such that \,  $\alpha_1,\dots, \alpha_n \models_k^{\leq} \alpha$.\\

\end{tabular}
\end{center}
\end{itemize}

\

From the above definition we have that  $\mathbb{L}_{k}^{\leq}$ is a sentential logic, that is, $\models_k^{\leq}$ is a finitary consequence relation on $Fm$. Besides: 

\begin{prop}\label{regla} Let \, $\alpha_1,\dots, \alpha_n, \alpha \in Fm$. Then, the following conditions are equivalent

\begin{itemize}
\item[\rm (i)] $\alpha_1,\dots, \alpha_n, \models_k^{\leq} \alpha$,
\item[\rm (ii)] the inequality $\alpha_1 \wedge \dots \wedge \alpha_n \preccurlyeq \alpha$ holds in the variety $\mathbb{C}_k$,

\item[\rm (iii)] for every $ (A,t)\in \mathbb{C}_k$ and every $h\in Hom_{\mathbb{C}_k}(Fm,(A,t))$, it is held that  $h(\alpha_1) \wedge \dots \wedge h(\alpha_n) \preccurlyeq h(\alpha)$.
\end{itemize}
\end{prop}
\begin{dem} (i)$\Rightarrow$(ii): \, Suppose that $\alpha_1,\dots, \alpha_n \models_k^{\leq} \alpha$. Then, $\forall A\in \mathbb{C}_k$, $\forall v\in Hom_{\mathbb{C}_k}(Fm,A)$, $\forall a\in A$, if $v(\alpha_i)\geq a$, for all $i\leq n$, then $v(\alpha)\geq a$.
Since $A$ is a lattice, we have that $v(\alpha_i)\geq a$, for all $i\leq n$ \, iff \,  $\bigwedge_{i=1}^{n}v(\alpha_i)\geq a$ \, iff \, $v\left(\bigwedge_{i=1}^{n}(\alpha_i)\right)\geq a$.
So, if $v\left(\bigwedge_{i=1}^{n}(\alpha_i)\right)\geq a$, then $v(\alpha)\geq a$; that is,  $\alpha_1 \wedge\dots \wedge \alpha_n  \preccurlyeq \alpha$.  \\[2mm]
(ii)$\Rightarrow$ (iii) and (iii)$\Rightarrow$ (i): \, Immediate.
\end{dem}

\

The logics of formal inconsistency (LFI) was introduced by Carnielli and Marcos in \cite{CM}. These systems were invented with the intention to generalize the da Costa's idea to have paraconsistent logics with a unary operator that allows us to recover the classical theorems. Recall that a logic ${\cal L}=\langle Fm,\vdash_{\cal L}\rangle$ is said to be strong LFI if  the following conditions hold:
\begin{itemize}
\item[\rm (i)] if $p$ and $q$ are two different propositional variables then:
\begin{itemize}
\item[\rm (i.a)] $p, \sim p \not\vdash_{\cal L} q$,
\item[\rm (i.b)]  $\circ p, p \not\vdash_{\cal L}q$,
\item[\rm (i.c)]  $\circ p, \sim p \not\vdash_{\cal L} q$,
\end{itemize}  
\item[\rm (ii)]  $\circ \varphi,\varphi,\sim\varphi\vdash_{\cal L}\psi$ for every $\varphi$ and $\psi$. 
\end{itemize}

The condition (i.a) for sentences is typically used to define paraconsistent logic. On the other hand, the  notion of Logic Formal  Undeterminedness (LFU)  was introduced by Marcos in \cite{Marcos}, they are paracomplete systems with a unary operator that permit us to recover  the Excluded Middle. That is to say, if the following holds: 
\begin{itemize}

\item[\rm (i)] $\not\vdash_{\cal L} \varphi \lor \sim\varphi$, for some $\varphi$; i.e., ${\cal L}$ is paracomplete;

\item[\rm  (ii)] there is a formula $\varphi$ such that 

\begin{itemize}
\item[\rm  (ii.a)] $\circ_k\varphi \not\vdash_{\cal L}  \varphi$;
\item[\rm  (ii.b)] $\circ_k\varphi \not\vdash_{\cal L} \sim \varphi$;
\end{itemize}  
\item[\rm (iii)] $\circ_k\varphi \vdash_{\cal L}  \varphi \lor \sim\varphi$, for every $\varphi$.
\end{itemize}

 In the rest of the section, we will show that  $\mathbb{L}^\leq_{k}$ is LFI and LFU. Now, let us consider a unary  operator $\circ_k$ for  $\mathbb{L}_{k}^{\leq}$ defined as follows:

 
 $$\circ \alpha = \bigwedge\limits_{i=0}^{k} t^i((\sim \alpha \vee \alpha)\wedge (\alpha \wedge \sim \alpha)^*)$$
 
 Now, let us consider the Boolean set of a given ${\cal C}_k$-algebra $(A,t)$, $B(A)=\{x\in A: x\wedge \sim x=0\}$. Thus, it is easy to see the following Proposition.

\begin{prop}
Let $(A,t)$ be a ${\cal C}_k$-algebra, then $x\in B(A)$ iff $\circ x =1$.
\end{prop}
\begin{proof}
Suppose that $x\in B(A)$, then  $\circ x=\bigwedge\limits_{i=0}^{k} t^i((\sim x \vee x)\wedge (x \wedge \sim x)^*)=1$. Conversely, let us suppose $\circ x =1$. Then, $\circ x= \bigwedge\limits_{i=0}^{k} t^i((\sim x \vee x)\wedge (x \wedge \sim x)^*)=1$ and so $\sim x \vee x=1$. 

\end{proof} 
 \begin{lem}
 $\mathbb{L}^\leq_{k}$ is a genuine paraconsistent logic. 
\end{lem}
\begin{proof}
Let us consider ${\cal C}_k$-algebra $(T_4, Id)$ where $T_4$ is the simple $mpM-$algebra and $id$ is the identity function. To see that (i.a) holds, it is enough to take the homomorphism $v:Fm \to (T_4, Id)$ such that $v(p)=a=v(\sim p)$ and $v(q)=0$, then $v(p\wedge\sim p)>v(q)$. Using the same valuation, we have $\not\models_k^{\leq} \sim( p\wedge \sim p)$, see \cite{Beziau2016} and \cite{OFP}.
\end{proof}

\begin{lem}\label{LFI}
The logic $\mathbb{L}^\leq_{k}$ is a strong LFI with respect to $\sim$ and $\circ$. 
\end{lem}
\begin{proof}
Let us consider ${\cal C}_k$-algebra $(T_4, Id)$ where $T_4$ is the simple $mpM-$algebra and $id$ is the identity function. To see that (i.a) holds, it is enough to take the homomorphism $v:Fm \to (T_4, Id)$ such that $v(p)=a=v(\sim p)$ and $v(q)=0$, then $v(p\wedge\sim p)>v(q)$. For the case (i.b), we get a function $v(\circ \top)=1$ $v(\top)=1$ and $v(q)=0$, then $v(\circ \top\wedge\top)>v(q)$. For (i.c), we get a function $v(\circ \bot)=1$ $v(\sim \bot)=1$ and $v(q)=0$, then $v(\circ \bot \wedge \sim\bot)>v(q)$

Now, for the case (ii), we can observe that the identity $((\sim x \vee x)\wedge (x \wedge \sim x)^*)\wedge x\wedge \sim x=0$ holds in $mpM-$ algebras $T_2$ and $T_4$. Then, since the class of $mpM-$algebra is a variety and the identity is in the language of $mpM-$algebra, we can affirm that the identity hold in every ${\cal C}_k$-algebra $(A,t)$,  then $v(\circ p \wedge p\, \wedge \sim p)\leq v(q)$ for every variable $p$ and $q$, which complete the proof.
\end{proof}

\begin{lem}\label{LFU}
The logic  $\mathbb{L}^\leq_{k}$ is an LFU with respect to $\sim$ and $\circ$. 
\end{lem}
\begin{proof}
Taking the same simple algebra $(T_4, Id)$ of Lemma \ref{LFI} and $v(p)=a=v(\sim p)$, we have that $\not\models^{\leq}_{k} p \vee \sim p$. Now, taking $\circ \bot \not\models^{\leq}_{k} \bot$ and $\circ \top \not\models^{\leq}_{k} \sim \top$. To see that $\circ \alpha\models^{\leq}_{k} \alpha \vee \sim \alpha$, it is enough to observe that the following algebraic property  $\circ x  \leq x\vee \sim x$ hold on every ${\cal C}_k$-algebra.
\end{proof}

\

Next, we will see that the operator $\circ$ enjoys the {\em propagation property} in the logic $\mathbb{L}_{k}$. That is to say:

\begin{lem}\label{PP}
In the logic $\mathbb{L}^\leq_{k}$ the following holds:

\begin{itemize}
\item[\rm (i)] $\models_k^{\leq}\circ \bot$,
\item[\rm (ii)] $\circ \alpha\models_k^{\leq}\circ \#_1 \alpha$ where $\#_1\in \{\sim, ^*, t\}$,
\item[\rm (iii)] $\circ \alpha, \circ\beta\models_k^{\leq}\circ (\alpha \# \beta)$ where $\#\in \{\wedge,\vee\}$.
\end{itemize} 
\end{lem}
\begin{proof}
It is not hard to see that the algebraic equations $\circ( 0)=1$, $ ((x\vee \sim x)\wedge (x\wedge \sim x)^*) \wedge  ((y\vee \sim y)\wedge (y\wedge \sim y)^*) \leq ((x \# y)\vee \sim (x \# y))\wedge ((x \# y)\wedge \sim x (\# y))^* $ (where $\#\in \{\wedge,\vee\}$), $\circ x = \circ (\sim x)$ and $\circ(x)=((x\vee \sim x)\wedge (x\wedge \sim x)^*) = ((x^*\vee \sim x^*)\wedge (x^*\wedge \sim x^*)^*)=\circ(x^*)$ hold in $T_3$ and $T_4$. Therefore, they are held in every $mpM-$algebra. So, taking in  mind the definition of $\circ$ and that $t$ is a homomorphism, we have that $\circ 1=1$, $\circ x \wedge  \circ y \leq \circ (x \# y)$ (where $\#\in \{\wedge,\vee\}$) and $\circ x = \circ (\#_1 x)$ (with $\#_1 \in \{\sim, ^*, t\}$) are held in every ${\cal C}_k$-algebra.
\end{proof}

\

For every positive integer $k$, we have that $\mathbb{L}^\leq_{k}$ is LFI and LFU, the same phenomenon occurs in the logics studied in \cite{CR19} and \cite{EFFG}, see also \cite{DS}.

\subsection{ $\mathbb{L}_{k}^{\leq}$ as a generalized matrix logic }

Recall that a {\em matrix logic} is a pair $\langle A, F\rangle$ where $A$ is an algebra and $F$ is a non-empty set of $A$. On the other hand, a {\em generalized matrix} is a pair $\langle A, {\cal C}\rangle$ where $A$ is an algebra and $\cal C$ is a family of subsets of $A$ that verify the following:

\begin{itemize}
\item[(i)] $A\in {\cal C}$,
\item[(ii)] If $\{X_i\}_{i\in I}$ then $\bigcup_{i\in I} X_i\in {\cal C}$.
\end{itemize}

Now, let us consider the matrix $M=\langle A, F\rangle$. Then the logic $\mathbb{L}^{k}_M=\langle Fm, \models_M\rangle$ is determined by $M$ in the following way:

let $\Gamma\cup\{\alpha\}$ a subset of $Fm$

\begin{center}
\begin{tabular}{ccc}
$\Gamma \models_M \alpha$ & $\Longleftrightarrow$ &   $\forall v\in Hom(Fm,A)$, \\
& & if $v(\gamma)\in F$, for all $\gamma\in \Gamma$, then $v(\alpha)\in F$ \\
\end{tabular}
\end{center}

If ${\cal M}=\{M_i\}_{i\in I}$ is a family of matrices, then the logic  $\mathbb{L}^{k}_{\cal M}=\langle Fm, \models_{\cal M}\rangle$  determined by the family $\cal M$ is defined by $\models_{\cal M}=\bigcap_{i\in I} \models_{M_i}$.

Let us consider now the logic determined by generalized matrix ${\cal G}=\langle A, {\cal C}\rangle$ that we denote 
 $\mathbb{L}^{k}_{\cal G}=\langle Fm, \models_{\cal G}\rangle$ and it is defined by the following:  let $\Gamma\cup\{\alpha\}$ a subset of $Fm$

\begin{center}
\begin{tabular}{ccc}
$\Gamma \models_{\cal G} \alpha$ & $\Longleftrightarrow$ &   $\forall v\in Hom(Fm,A)$, $\forall G\in \cal C$ \\
& & if $v(\gamma)\in G$, for all $\gamma\in \Gamma$, then $v(\alpha)\in G$ \\
\end{tabular}
\end{center}

In this case, we take $${\cal M}=\{\langle A, F\rangle: A\in {\mathbb{C}_k},\,\,\text{and}\,\, F\,\,\text{is a lattice filter of}\,\, A\} $$

and 
$${\cal G}=\{\langle A, {\cal F}_i(A)\rangle: A\in {\mathbb{C}_k}\}$$

where ${\cal F}_i(A)$ is the family of all filters of $A$. Then, we have the following Lemma:

\begin{lem}
The logic $\mathbb{L}_{k}^{\leq}$ coincides with both logics $\mathbb{L}^{k}_{\cal M}$ and $\mathbb{L}^{k}_{\cal G}$.
\end{lem}

\begin{proof}
The fact that the logics $\mathbb{L}^{k}_{\cal M}$ and $\mathbb{L}^{k}_{\cal G}$ coincide follows immediately from their very definition of them. On the other hand, in  the variety $\mathbb{C}_k$ the notion of filter for every algebra is elementary definable, then the class ${\cal M}$ is closed under ultraproducts. Therefore, $\mathbb{L}^{k}_{\cal M}$ is finitary.

Now, we will see that the systems $\mathbb{L}_{k}^{\leq}$ and  $\mathbb{L}^{k}_{\cal G}$ coincide; that is to say, $\models_{k}^{\leq}=\models_{\cal G}$. Since $\models_{\cal G}$ is finitary, we make the proof using a finite set of premises. Indeed, let $\alpha_1,\cdots, \alpha_n, \alpha$ be a set of formulas such that $\alpha_1,\cdots, \alpha_n \models_{k}^{\leq} \alpha$. According to Proposition \ref{regla}, for every  $v\in Hom_{\mathbb{C}_k}(Fm,A)$ we have that $v(\alpha_1)\wedge \cdots \wedge v(\alpha_n)\leq v(\alpha)$. Let us consider now $F\in {\cal F}_i(A)$ and suppose that $v(\alpha_1), \cdots, v(\alpha_n)\in F$. Since $F$ is a filter and taking in mind the definition of $\mathbb{L}^{k}_{\cal G}$, we have that $\alpha_1,\cdots, \alpha_n \models_{\cal G} \alpha$.

Conversely, suppose that $\alpha_1,\cdots, \alpha_n \models_{\cal G} \alpha$ and let  $v\in Hom_{\mathbb{C}_k}(Fm,A)$. Now, taking a filter generated $G$ by the set $\{ v(\alpha_1),\cdots, v(\alpha_n)\}$, we have that $v(\alpha_i)\in G$ with $i=1,\cdots,n$. Then, $v(\alpha)\in G$. Again taking into account Proposition \ref{regla}, we obtain that $\alpha_1,\cdots, \alpha_n \models_{k}^{\leq} \alpha$ as desired.
\end{proof}

\subsection{Logic $\mathbb{L}_{k}^{\leq}$ location   inside the hierarchy of Abstract algebraic logic } 

The theory of {\em algebrizable logics} was initiated by Blok and Pigozzi in \cite{BP}, in that  work the notion of {\em Leibniz operator} to characterize them was introduced. Other weak notions of algebrizable logic can be characterized by the study of  Leibniz operator.  Now, we will summarize the main background to develop the rest of the section.

If $\langle {\bf A}, F\rangle$ is an arbitrary matrix, its {\em Leibniz congruence} ${\bf \Omega_A} F$ is the largest of all congruences of $\bf A$ that are compatible with $F$ in the following sense:

$${\bf \Omega_A} F:=\{\Theta\in {\rm Co}{\bf A}: \text{if}\,\, (a,b)\in \Theta\,\, \text{and}\,\, a\in F \,\, \text{then}\,\, b\in F\}$$

Where ${\rm Co}\bf A$ is the poset of all congruences of $\bf A$. Besides, a matrix is said to be {\em \bf reduced} when ${\bf \Omega_A} F$ is the identity; that is to say, the identity relation is only congruence compatible with $F$.

To define the notion of {\em algebraic counterpart of a logic $L$}, we need to consider the class of $Alg^*(L)$ of algebras reducts of the reduced models of $L$ as follows:

$$ Alg^*(L):=\{{\bf A}: \exists F\subseteq A, \langle A, F\rangle \, \, \text{is a model of }\,\, F\,\, \text{and}\,\, {\bf \Omega_A} F=Id\}  $$

The logic $L$ is algebraizable, then $Alg^*(L)$ coincides with its equivalent algebraic semantics. This is the case of Rasiowa's {\em implicative models} treated in \cite{RA}, see also \cite[Lemma 2.6]{Font}.

The theory to associate generalized matrix model with a class of algebra with an arbitrary logic can be found in \cite{FJ}. Indeed, if $\langle A, {\cal C}\rangle$ is a generalized matrix, then its Tarski congruence $ \widetilde{\bf \Omega}_{\bf A} F$ is the largest congruence of $\bf A$ that is compatible with all members of $\cal C$; it is not hard to show that:

$$ \widetilde {\bf \Omega}_{\bf A} F=\bigcap\limits_{F\in \cal C}  {\bf \Omega_A} F$$

A generalized matrix is {\em reduced} when  Tarski congruence is the identity and the class of $L$-algebras $Alg(L)$ is defined as the class of reducts of the reduced generalized matrix if $L$, i.e.:

$$ Alg(L):=\{{\bf A}: \exists {\cal C} \subseteq P(A), \langle A, {\cal C}\rangle \, \, \text{is a generalized model of }\,\, F\,\, \text{and}\,\, \widetilde{{\bf \Omega}}_{\bf A} {\cal C}=Id\}  $$

It is well-known that $Alg(L)$ is the class of subdirect products of algebras in  $ Alg^*(L)$,  and in general $Alg(L)\subseteq  Alg^*(L)$ holds, see \cite[Theorem 5.70]{Font}. One of the main aims of this section is to see that $Alg(\mathbb{L}_{k}^{\leq})=\mathbb{C}_k$. To this end, from Proposition \ref{regla} we can infer that $\models_{k}^{\leq}$ is a {\em semilattice based} with respect to $\mathbb{C}_k$ in the Jansana's sense \cite[p. 76]{RJ}.

\begin{proposition}\label{algebras}
$Alg(\mathbb{L}_{k}^{\leq})=\mathbb{C}_k$.
\end{proposition}

\begin{proof}
From Proposition 3.1 \cite{RJ}, we have that $Alg(\mathbb{L}_{k}^{\leq})$ is the intrinsic variety of $L$ which is characterized by the equations $\psi \approx \psi$. Now, from Proposition \ref{regla} we have that  $\psi \approx \psi$ iff $\phi\models_{k}^{\leq} \psi$ and  $\psi\models_{k}^{\leq} \phi$ which implies the identities that characterized $\mathbb{C}_k$ are exactly the same that the one of  $Alg(\mathbb{L}_{k}^{\leq})$.
\end{proof}

\

Recall that logic $L$ is said to be {\em selfextensional} if $L$ satisfies the following {\em weak form of the replacement property}: For any $\alpha,\beta,\psi(x,\vec{y})\in Fm$,

\begin{center}
if $\alpha \dashv\vdash_L \beta$ then $\psi(\alpha,\vec{y}) \dashv\vdash_L  \psi(\beta,\vec{y})$.
\end{center}

\begin{coro}
$\mathbb{L}_{k}^{\leq}$ is  selfextensional.
\end{coro}
\begin{proof}
It is an immediate consequence from \cite[Theorem 3.2]{RJ} and Proposition \ref{algebras}.
\end{proof}

\

Now, we start with the study of the classification of $\mathbb{L}_{k}^{\leq}$ in the Leibniz hierarchy. This hierarchy is characterized in several ways, although the majority  concerns several properties of Leibniz operator given by:

$$F \longmapsto  {\bf \Omega_A} F$$

A logic $L$ is  {\em algebraizable} (\cite{BP}) when for every matrix $\langle {\bf A}, F\rangle$ which is model of $L$ for every algebra $\bf A$ belonging the quasivariety $K$, the Leibniz operator ${\bf \Omega_A} F$ is an isomorphism between the lattice ${\cal F}_i({\bf A})$ and ${\rm Co}{\bf A}$. Besides, it is possible to see that $Alg^*(L)=K$.

A logic $L$ is a {\em protoalgebraic} one when Leibniz operator is monotonic; i.e.,  if $G\subseteq F$ then ${\bf \Omega_A} G\subseteq {\bf \Omega_A} F$. It is clear that every algebraizable logic is protoalgebraic. An intermediate class of logics between them is of the {finitely equivalential} ones. Recall that a logic $L$ is said to be {\em equivalential} if there is a set of formulas built-up in two variables $p$ and $q$, $E(p,q)$ such that the following conditions hold:

\begin{itemize}
\item[(E1)] $\vdash_L \delta(p,p)$ for every $\delta(p,q)\in  E(p,q)$,
\item[(E2)] $E(p,q) \cup \{p\} \vdash_L  q$,
\item[(E3)] for every $n$-ary operation $\#$ and two $n$-tuples of variables $q_1,\cdots,q_n, p_1,\cdots,p_n$, we have: $E(p_1,q_1) \cup \cdots E(p_n,q_n) \vdash_L E(\#(p_1,\cdots,p_n),\#(q_1,\cdots,q_n))$.
\end{itemize}

We say that $L$ is {\em finitely equivalential} if the set $E(p,q)$ is finite. To see that the logic $\mathbb{L}_{k}^{\leq}$ is finitely equivalential, we will consider the following implication: 

$$x\Rightarrow y:= (x\to y)\wedge (\sim y \to \sim x)$$

where $x\to y:= \sim ((\sim (x\vee y))^* \wedge (x\vee y)) \vee (( \sim (x\wedge y))^* \wedge (x\wedge y))$. The following Proposition presents properties of this implication to be used in the rest of the paper.

\begin{proposition}\label{propiedades}
For a given ${\cal C}_k$-algebras $(A,t)$ and any $x,y,z,w \in A$, the following holds:
\begin{multicols}{2}
\begin{itemize}
\item[\rm (i)] $x\Rightarrow y= y\Rightarrow x$,
\item[\rm  (ii)] $x\Rightarrow y= \sim x \Rightarrow \sim y$,
\item[\rm (iii)] $x\Rightarrow x=1$,
\item[\rm  (iv)] $(x\Rightarrow y) \wedge (y\Rightarrow z)\leq (x\Rightarrow z)$, 
\item[\rm (v)] $(x\Rightarrow y) \wedge (z\Rightarrow w)\leq (x \wedge z) \Rightarrow (y\wedge w) $, 
\item[\rm (vi)] $(x\Rightarrow y) \wedge (z\Rightarrow w)\leq (x \vee z) \Rightarrow (y\vee w) $, 
\item[\rm  (vii)] $x\Rightarrow y = x^* \Rightarrow y^*$,

\item[\rm (viii)] $(x\Rightarrow y)\wedge x = (x\Rightarrow y)\wedge y$,
 
\item[\rm (ix)]  $(1\Rightarrow x)=1$ iff $x=1$.
\end{itemize}
\end{multicols}
\end{proposition}

\begin{proof} Let us consider ${\cal C}_k$-algebra $(A,t)$. Then, it is clear that $A$ is an $mpM-$algebra, see Section \ref{preli}. Since the variety of $mpM-$algebra is generated by $T_3$ and $T_4$, every identity that is verified in the generating algebras  holds in every algebra of the class. The identities expressed in (i) to (viii) are in the language of the $mpM-$algebras and so it is easy to see that using the corresponding  table, these identities hold in  $(A,t)$.

Now, let us suppose that $(1\Rightarrow x)=1$. Then, $(1\Rightarrow x)= (1\to x)\wedge (\sim x\to \sim 1)= \sim ((\sim 1)^* \wedge 1) \vee (( \sim x)^* \wedge x)\wedge  \sim ((\sim (\sim x)^* \wedge (\sim x)) \vee (( \sim 0)^* \wedge 0)= (( \sim x)^* \wedge x)\wedge (  (\sim x)^* \vee  x))=1$. The latter identity is equivalent to $(( \sim x)^* \wedge x)=1$ and $( (\sim x)^* \vee  x))=1$  iff $( \sim x)^*=1$ and $ x=1$, which completes the proof.
\end{proof}

\begin{definition}
Let $\delta(p,q)$ be the following formula of $\mathcal{L}_{\Sigma}$:
$$\delta(p,q)=(p \Rightarrow q)\wedge ( t p \Rightarrow t q)\wedge( t^2 p \Rightarrow t^2 q)\wedge \cdots \wedge ( t^{k-1} p \Rightarrow t^{k-1} q).  $$
The   equivalence sentences system for $\mathbb{L}_{k}^{\leq}$ is  $E(p,q)=\{\delta(p,q)\}$.
\end{definition}

\begin{theorem}\label{proto}
The logic $\mathbb{L}_{k}^{\leq}$ is protoalgebraic and finitely equivalential, but not algebraizable.
\end{theorem}
\begin{proof}
The fact  that $\models_{k}^{\leq}$ satisfies (E1) is a consequence of (iii) from Proposition \ref{propiedades}. From the definition of $\models_{k}^{\leq}$  and (viii) of Proposition \ref{propiedades}, we have that (E2) is held. Now, the fact that  $\models_{k}^{\leq}$ verifies (E3) follows from (ii), (v), (vi) and (vii) of Proposition \ref{propiedades}. Since the logic is finitely equivalential, then it is protoalgebraic. To see that $\mathbb{L}_{k}^{\leq}$ is not algebraizable is enough to take the simple algebra $(T_4, Id)$ where $T_4$ is the simple $mpM-$algebra and $id$ is the identity function. It is not difficult to see that if we take the filters $\uparrow a=\{a,1\}$ and $\uparrow b=\{b,1\}$, then ${\bf \Omega_A} \uparrow a={\bf \Omega_A} \uparrow b=A\times A$, then the  Leibniz operator is not 1-1, by which the proof is completed.
\end{proof}

\

From the proof of the last Theorem, it is clear that the logic $\mathbb{L}_{k}^{\leq}$ is not weak algebraizable because the Leibniz operator is not monotonic.

\begin{coro}\label{counter1}
$Alg^*(\mathbb{L}_{k}^{\leq})=\mathbb{C}_k$.
\end{coro}

\begin{proof}
From  Theorem \ref{proto} and  \cite[Proposition 3.2]{FJ}, we have that $Alg^*(\mathbb{L}_{k}^{\leq})=Alg(\mathbb{L}_{k}^{\leq})$. From the latter and Proposition \ref{algebras}, we have completed the proof.
\end{proof}
\section{ Another  sentential logic associated with the class ${\cal C}_k$-algebras}

In this section, we will present another sentential logic that has an algebraic counterpart the class ${\cal C}_k$-algebras. To this end, let us consider $\mathfrak{Fm}=\langle Fm, \wedge, \vee, \sim, ^\ast, t, \top, \bot \rangle$ the absolutely free algebra of formulas of type $(2,2,1,1,1,0)$ generated by some denumerable set of variables. Now we will say that $\mathbb{L}_{k}=\langle Fm, \models_k\rangle$ is a sentential logic which is defined as follows  for every $\Gamma\cup\{\alpha\}\subseteq Fm$:



\begin{tabular}{ccc}
$\Gamma \models_k \alpha$ & $\Longleftrightarrow$ &  $\forall (A,t)\in \mathbb{C}_k$, $\forall v\in Hom_{\mathbb{C}_k}(Fm,(A,t))$\\
& & if $v(\gamma)=1$ for every $\gamma\in\Gamma$ then $v(\alpha)= 1$ \\
\end{tabular}\\[2mm]

In the area of algebraic logic, these kind of  logics defined as above are so-called $1$-assertional ones. When the logic is built over a lattice, it is generally finitary. Indeed:
 
\begin{proposition}
The logic $\mathbb{L}_{k}$ is finitary.
\end{proposition} 
\begin{proof} It  immediately follows from the fact that $\mathbb{C}_k$ is a variety and  the notion of filter for every algebra is elementary definable and then the class ${\cal C}_k$-algebras is closed under ultraproducts.
 \end{proof}

\begin{proposition}
The logics $\mathbb{L}_k$ and $\mathbb{L}_{k}^{\leq}$ have the same theorems.
\end{proposition}
\begin{proof}
It immediately follows from the very definitions of $\models_k$ and $\models_{k}^{\leq}$.
\end{proof}

\begin{lem}\label{parac}
The logic $\mathbb{L}_k$ is paracomplete, but not paraconsistent.
\end{lem}
\begin{proof}
Taking the simple ${\cal C}_k$-algebra $(T_4,id)$ and $v(p)=a$ and $v(\sim p)=a$, then we have that $\not\models_k p \vee \sim p$. Now, from the fact that the negation $\sim$ is involutive in every ${\cal C}_k$-algebra, we have that $p,\sim p \models_k q$ holds for every variable $p$ and $q$.

\end{proof}

\begin{lem}
The logic $\mathbb{L}_k$ is a  LFU with respect to $\sim$ and $\circ$ (defined in the above section).  
\end{lem}
\begin{proof}
It is an immediately consequence from Lemmas \ref{parac} and  \ref{LFU}.
\end{proof}

\

On the other hand, Blok and Pigozzi developed the theory of algebraizable logics in \cite{BP} using a generalization of the usual Lindenbaum-Tarski  process.  To present a characterization  of  algebraizable,  we will consider the propositional signature, $\Theta$, and let $\mathscr{L}$ be a standard propositional logic defined over the language $\mathcal{L}_{\Theta}$, with a consequence relation $\vdash_{\mathcal{L}}$. Then $\mathscr{L}$ is algebraizable in the sense of Blok and Pigozzi if there exists a nonempty set $\Delta(p_1,p_2)\subseteq\mathcal{L}_{\Theta}$ of formulas depending on variables $p_1$ and $p_2$, and a nonempty set $E(p_1)\subseteq\mathcal{L}_{\Theta}\times\mathcal{L}_{\Theta}$ of pairs of formulas depending on variable $p_1$ satisfying the following properties:
\begin{itemize}
    \item[\rm (C1)] $\vdash_{\mathscr{L}}\delta(p_1,p_1)$, for every $\delta(p_1,p_2)\in\Delta(p_1,p_2)$;
   
    \item[\rm (C2)] $\Delta(p_1,p_2)\vdash_{\mathscr{L}}\delta(p_2,p_1)$, for every $\delta(p_1,p_2)\in\Delta(p_1,p_2)$;
   
    \item[\rm (C3)] $\Delta(p_1,p_2)$, $\Delta(p_2,p_3)\vdash_{\mathscr{L}}\delta(p_1,p_3)$, for every $\delta(p_1,p_2)\in$ $\Delta(p_1,p_2)$;
   
    \item[\rm (C4)] $\Delta(p_1,p_{n+1}),\ldots,\Delta(p_n,p_{2n})\vdash_{\mathscr{L}}\delta(\#(p_1,\ldots,p_n),\#(p_{n+1},\ldots,p_{2n}))$, for every $\delta(p_1,p_2)\in\Delta(p_1,p_2)$, every n-ary connective $\#$ of $\Theta$ and every $n\geq 1;$
   
    \item[\rm (C5)] $p_1\vdash_{\mathscr{L}}\delta(\gamma(p_1),\epsilon(p_1))$, for every $\delta(p_1,p_2)\in \Delta(p_1,p_2)$ and every\\
    $\langle \gamma(p_1),\epsilon(p_1) \rangle\in E(p_1)$;
   
    \item[\rm (C6)] $\big\{\delta(\gamma(p_1),\epsilon(p_1)): \delta(p_1,p_2)\in\Delta(p_1,p_2), \langle\gamma(p_1),\epsilon(p_1)\rangle\in E(p_1) \big\}\vdash_{\mathscr{L}}p_1$.
\end{itemize}
The sets $\Delta(p_1,p_2)$ and $E(p_1)$ are called system of equivalence formulas and defining equations, respectively. In order to see that the logic  $\mathbb{L}_{k}$ is algebraizable, we will consider the following two binary operations defined on the above section: 

$$x\to y:= \sim ((\sim (x\vee y))^* \wedge (x\vee y)) \vee (( \sim (x\wedge y))^* \wedge (x\wedge y))$$

$$x\leftrightarrow y:= (x\to y)\wedge (\sim y \to \sim x)$$

Now, let us take formula $\delta(p_1,p_2)$ defined as follows:
$$\delta(p_1,p_2)=(p_1\leftrightarrow p_2)\wedge ( t p_1 \leftrightarrow t p_2)\wedge( t^2 p_1 \leftrightarrow t^2 p_2)\wedge \cdots \wedge ( t^{k-1} p_1 \leftrightarrow t^{k-1} p_2)  $$
The  system of equivalence formulas is given by $\Delta(p_1,p_2)=\{\delta(p_1,p_2)\}$ and the system of defining equations is given by $E(p_1)=\{p_1,p_1 \leftrightarrow p_1\}$.

\begin{theorem}\label{alge}
The logic $\mathbb{L}_{k}$ is algebraizable in the sense of Blok-Pigozzi  with the system of equivalence formulas  $\Delta(p_1,p_2)$ and the system of defining equations $E(p_1)=\{p_1, p_1\leftrightarrow  p_1\}$.
\end{theorem}
\begin{proof}
We have to prove that $\models_k$ satisfies (C1) to (C6).  The fact that (C1) and (C2) hold follows immediately from Proposition \ref{propiedades} (iii) and (i), respectively.

\noindent (C3): Let us suppose that for every $v\in Hom_{\mathbb{C}_k}(Fm,(A,t))$ we have: 

$$v((p_1\leftrightarrow p_2)\wedge ( t p_1 \leftrightarrow t p_2)\wedge( t^2 p_1 \leftrightarrow t^2 p_2)\wedge \cdots \wedge ( t^{k-1} p_1 \leftrightarrow t^{k-1} p_2)=1$$

and 

$$v((p_2\leftrightarrow p_3)\wedge ( t p_2 \leftrightarrow t p_3)\wedge( t^2 p_2 \leftrightarrow t^2 p_3)\wedge \cdots \wedge ( t^{k-1} p_2 \leftrightarrow t^{k-1} p_3)=1$$

Then, it is clear that $v(p_1\leftrightarrow p_2)=1$ and $v(p_2\leftrightarrow p_3)=1$; moreover, $v(t^i p_1\leftrightarrow t^i p_2)=1$ and $v(t^i p_2\leftrightarrow t^i p_3)=1$ for any $i=1,\cdots, k-1$. From the latter and Proposition \ref{propiedades} (iv), $v(p_1\leftrightarrow p_3)=1$ and  $v(t^i p_1\leftrightarrow t^i p_3)=1$ for any $i=1,\cdots, k-1$. 

\noindent (C4): Suppose  for any $v\in Hom_{\mathbb{C}_k}(Fm,(A,t))$ and the following identities  $v(\delta(p_1,p_2))=1$ and  $v(\delta(p_3,p_4))=1$ hold. So, $v(t^i p_1\leftrightarrow t^i p_2)=1$ and $v(t^i p_3\leftrightarrow t^i p_4)=1$ for any $i=1,\cdots, k$. Now from the items (v) and (iv) of Proposition \ref{propiedades}, we have that  $v(t^i p_1 \# t^i p_3 \leftrightarrow t^i p_2 \# t^i p_4 )=1$ and so $v(\delta(p_1\# p3,p_2\# p_4))=1$ where $\#\in \{\wedge,\vee\}$. Now it is not hard to see that if $v(\delta(p_1,p_2))=1$, then $v(\delta(\#p_1,\#p_2))=1$ with $\#\in\{^*, \sim,t\}$.

\noindent (C5): Suppose  for any $v\in Hom_{\mathbb{C}_k}(Fm,(A,t))$ we have $v(p_1)=1$, then it is not hard to see that $v(\delta(p_1,p_1))=1$ and $v(\delta(p_1,p_1 \leftrightarrow p_1))=1$.

\noindent (C6): Suppose  for any $v\in Hom_{\mathbb{C}_k}(Fm,(A,t))$, we have $v(\delta(p_1,p_1))=1$ and $v(\delta(p_1,p_1 \leftrightarrow p_1))=1$. So, $v( p_1\leftrightarrow p_1)=1$ and then we have $1 \leftrightarrow v(p_1)=1$. From Proposition \ref{propiedades} (viii), we have that $v(p_1)=1$ as desired.
\end{proof}

\

As consequence of the last Theorem and some general results of the Blok-Pigozzi theory, we have the following:

\begin{coro}\label{counter2}
$Alg^*(\mathbb{L}_{k})=\mathbb{C}_k$.
\end{coro}
\begin{proof}
See section 3 and  Proposition 3.15 of \cite{Font}.
\end{proof}

\

It is important to note that $\mathbb{L}_{k}$ is not Rasiowa's {implicative logic} using $\delta(p_1,p_2)$ as implication, see \cite[Definition 2.5]{Font} because  it is not verified:  if $v((p_1\leftrightarrow p_2)\wedge ( t p_1 \leftrightarrow t p_2)\wedge( t^2 p_1 \leftrightarrow t^2 p_2)\wedge \cdots \wedge ( t^{k-1} p_1 \leftrightarrow t^{k-1} p_2)=1$, then $v(p_1)=v(p_2)$ for every ${\cal C}_k$-algebra.

\section{Conclusions} 

In this paper, we have studied the class of $k$-cyclic modal pseudocomplemented De Morgan algebras proving that it is a semsisimple variety and determining the generating algebras for every positive integer $k$. From the algebraic studies, we have displayed a family of sentential logics $\mathbb{L}^\leq_{k}$ and $\mathbb{L}_{k}$ that have as algebraic counterpart the class of algebras introduced previously, this fact has been settled in Corollaries \ref{counter1} and \ref{counter2}. As the main results that emerge from our studies, it has been shown that although the logics $\mathbb{L}^\leq_{k}$ and $\mathbb{L}_{k}$ share the same theorems,  $\mathbb{L}^\leq_{k}$ is a paraconsistent non-algebraizable logic; in contrast, $\mathbb{L}_{k}$  is a non-paraconsistent algebraizable logic.


\

\

\noindent  A. Figallo-Orellano\\
Departamento de  Matem\'{a}tica\\
Universidad Nacional del Sur (UNS)\\
Bah\'ia Blanca, Argentina\\
e-mail:  {\em  aldofigallo@gmail.com}

\

\

\noindent M. P\'erez-Gaspar\\
Facultad de Ingenirer\'ia\\
Universidad Nacional Aut\'onoma de M\'exico (UNAM)\\
Ciudad de  M\'exico, M\'exico\\
e-mail:  {\em miguel.perez@fi-b.unam.mx}

\

\

\noindent J. M. Ram\'irez-Contreras \\
Departamento de Actuar\'{\i}a, F\'{\i}sica y Matem\'{a}ticas\\
Universidad de las Am\'ericas Puebla (UDLAP)\\
Puebla, M\'exico\\
e-mail: {\em juan.ramirez@udlap.mx}

\end{document}